\documentclass[preprint,12pt]{elsarticle}

\usepackage{amsmath,amsfonts,amssymb}
\usepackage{algorithmic}
\usepackage{algorithm}
\usepackage{array}
\usepackage[dvipsnames]{xcolor}
\usepackage[caption=false,font=normalsize,labelfont=sf,textfont=sf]{subfig}
\usepackage{textcomp}
\usepackage{stfloats}
\usepackage{url}
\usepackage{verbatim}
\usepackage{graphicx}

\usepackage{amsmath} 
\usepackage{amssymb} 
\usepackage{amsthm} 
\usepackage[hidelinks]{hyperref} 
\usepackage{mathtools} 

\usepackage{cleveref} 

\usepackage{siunitx} %
\usepackage{pgfplots}
\usepackage{tuda-pgfplots}
\usepackage[normalem]{ulem} 
\usepackage{circuitikz}

\newtheorem{theorem}{Theorem}
\newtheorem{remark}[theorem]{Remark}
\newtheorem{lemma}[theorem]{Lemma}

\newtheorem{assumption}{Assumption}

\newcommand{\Ac}{\boldsymbol{A}_{\mathrm{C}}}
\newcommand{\Afoil}{\boldsymbol{A}_{\mathrm{foil}}}
\newcommand{\Ai}{\boldsymbol{A}_{\mathrm{I}}}
\newcommand{\Al}{\boldsymbol{A}_{\mathrm{L}}}
\newcommand{\Ar}{\boldsymbol{A}_{\mathrm{R}}}
\newcommand{\Asrc}{\boldsymbol{A}_{\mathrm{src}}}
\newcommand{\Astr}{\boldsymbol{A}_{\mathrm{str}}}
\newcommand{\Av}{\boldsymbol{A}_{\mathrm{V}}}
\newcommand{\alphadomain}{L_\alpha}

\newcommand{\Bmat}{\mathcal{B}}

\newcommand{\capacitancematrix}{\boldsymbol{C}}
\newcommand{\conductancematrix}{\boldsymbol{G}}
\newcommand{\controlinput}{\boldsymbol{u}}
\newcommand{\controloutput}{\boldsymbol{y}}
\newcommand{\coordinatetransform}{\boldsymbol{f}}
\newcommand{\current}{\imath}
\newcommand{\currentvector}{\boldsymbol{\imath}}
\newcommand{\cvector}{\boldsymbol{c}}
\DeclareMathOperator{\Diag}{Diag}
\newcommand{\diff}{\mathrm{d}}
\newcommand{\discretescalarpotential}{\boldsymbol{e}}
\newcommand{\discretevectorpotential}{\boldsymbol{a}}
\newcommand{\discretesourcecurrent}{\boldsymbol{j}_{\mathrm{s}}}

\newcommand{\effort}{\boldsymbol{e}} 
\newcommand{\Emat}{\mathcal{E}}
\newcommand{\FmatSkew}{\mathcal{F}_\text{skew}}
\newcommand{\FmatSym}{\mathcal{F}_\text{sym}}
\newcommand{\finaltime}{t_{\mathrm{f}}}
\newcommand{\foilthickness}{b}
\newcommand{\hamiltonian}{\mathcal{H}}
\newcommand{\inductancematrix}{\boldsymbol{L}}
\newcommand{\jl}{\boldsymbol{\jmath}_{\mathrm{L}}}
\newcommand{\jsol}{\boldsymbol{\jmath}_{\mathrm{sol}}}
\newcommand{\jsrc}{\boldsymbol{\jmath}_{\mathrm{src}}}
\newcommand{\jv}{\boldsymbol{\jmath}_{\mathrm{V}}}
\newcommand{\Jmat}{\mathcal{J}}
\newcommand{\Mmat}{\mathcal{M}}
\newcommand{\massmatrix}{\boldsymbol{M}_\sigma}

\newcommand{\np}{n_p}

\newcommand{\nw}{n_w}
\newcommand{\nstr}{n_\text{str}}
\newcommand{\nsol}{n_\text{sol}}
\newcommand{\nfoil}{n_\text{foil}}
\newcommand{\pbasisfunction}{p}
\newcommand{\R}{\mathbb{R}}
\newcommand{\resistancematrix}{\boldsymbol{R}}
\newcommand{\Rmat}{\mathcal{R}}
\newcommand{\state}{\boldsymbol{z}}
\newcommand{\stiffnessmatrix}{\boldsymbol{K}_\nu}
\newcommand{\transpose}{^\top}
\newcommand{\vertexpotentials}{\boldsymbol{\phi}}
\newcommand{\voltage}{v}
\newcommand{\voltagevector}{\boldsymbol{v}}
\newcommand{\wbasisfunction}{\vec{w}}
\newcommand{\Xsigma}{\boldsymbol{X}_\text{foil}}
\newcommand{\Xsol}{\boldsymbol{X}_\textrm{sol}}
\newcommand{\Xstr}{\boldsymbol{X}_\textrm{str}}
\newcommand{\xvector}{\boldsymbol{x}}

\newcommand{\ddt}{\frac{\diff}{\diff t}}
\newcommand{\tddt}{\tfrac{\diff}{\diff t}}

\newcommand{\integral}[4]{\int\limits_{#1}^{#2} #3\,\diff #4}
\newcommand{\norm}[1]{\bigl\lVert#1\bigr\lVert}

\journal{Applied Mathematical Modelling}

\begin{document}

\begin{frontmatter}

\title{Energy-based modeling for field--circuit coupling}

\author[umag]{Robert Altmann}
\author[tue]{Idoia Cortes Garcia}
\author[tuda,tau]{Elias Paakkunainen\corref{cor1}}
\author[tube,upot]{Philipp~Schulze}
\author[tuda]{Sebastian Schöps}

\cortext[cor1]{Corresponding author: Elias Paakkunainen (e-mail: elias.paakkunainen@tu-darmstadt.de).}

\affiliation[umag]{organization={Institute of Analysis and Numerics, Otto von Guericke University Magdeburg},
            city={Magdeburg},
            postcode={39106}, 
            country={Germany}}

\affiliation[tue]{organization={{Department of Mechanical Engineering, Eindhoven University of Technology}},
            city={{Eindhoven}},
            postcode={{5600 MB}}, 
            country={{The Netherlands}}}

\affiliation[tuda]{organization={Institute for Accelerator Science and Electromagnetic Fields, Technical University of Darmstadt},
            city={Darmstadt},
            postcode={64289}, 
            country={Germany}}

\affiliation[tau]{organization={Electrical Engineering Unit, Tampere University},
            city={Tampere},
            postcode={33720}, 
            country={Finland}}

\affiliation[tube]{organization={Institute of Mathematics, Technische Universität Berlin},
			city={Berlin},
			postcode={10623}, 
			country={Germany}}

\affiliation[upot]{organization={Institute of Mathematics, University of Potsdam},
			city={Potsdam},
			postcode={14476}, 
			country={Germany}}

\begin{abstract}
This paper presents a generalized energy-based modeling framework extending recent formulations tailored for differential--algebraic equations. The proposed structure, inspired by the port-Hamiltonian formalism, ensures passivity, preserves the power balance, and facilitates the consistent interconnection of subsystems. A particular focus is put on low-frequency power applications in electrical engineering. Stranded, solid, and foil conductor models are investigated in the context of the eddy current problem. Each conductor model is shown to fit into the generalized energy-based structure, which allows their structure-preserving coupling with electrical circuits described by modified nodal analysis. Theoretical developments are validated through a numerical simulation of an oscillator circuit, demonstrating energy conservation in lossless scenarios and controlled dissipation when eddy currents are present. The applicability of the methodology towards engineering applications is studied through a numerical simulation of a nonlinear three-phase transformer.
\end{abstract}


\end{frontmatter}

\section{Introduction}
\label{sec:intro}

In recent years, energy-based modeling via port-Hamiltonian (pH) formulations has received more and more attention; cf.~\cite{MehU23,SchJ14} for a general overview.
The pH framework generalizes classical Hamiltonian systems by allowing for energy dissipation as well as energy exchange with the environment.
Moreover, the corresponding structure implies a power balance and ensures passivity, i.e., the system may not internally generate energy.
Often, the power balance also implies (Lyapunov) stability, e.g., if the Hamiltonian is a squared norm of the state.
Furthermore, power-preserving interconnections of pH systems result in an overall system which is again pH. This makes the framework especially suitable for network modeling~\cite{FiaZOSS13,Sch04} and control \cite{MalRSH23,Sch20}. PH structures have been identified in various application areas, 
including chemistry \cite{HoaCJL11,TefDP22},
Maxwell equations \cite{Farle_2013aa,Haine_2022aa}, 
electrical engineering \cite{GerHRS21,Bartel_2024aa,CleHKG24}, 
continuum mechanics \cite{MacM04,WarBST21,RasS25},
and fluid mechanics \cite{AltS17}.
Moreover, structure-preserving discretization \cite{KotL19,BruRS22,GieKT24} and model order reduction techniques \cite{BreMS22,GugPBS12} play an important role.

Recently, a new representation of energy-based models has been introduced in \cite{AltS25} that appears to be especially suitable for differential--algebraic equation systems (DAE).
Similarly to classical pH formulations, it implies a power balance, is preserved under power-preserving interconnections, and allows for structure-preserving discretization and model reduction.
Moreover, in some applications, it allows to obtain energy-based formulations with fewer state variables compared to the pH-DAE formulations presented in \cite{BeaMXZ18,MehM19}.
In addition, the new framework also includes systems of index 3, whereas the linear pH-DAE formulation in \cite{BeaMXZ18} is restricted to systems with an index of at most 2, cf.~\cite{MehMW18}. 

This paper introduces a generalization of the original energy-based framework to account for practically relevant models from electrical engineering, i.e., modified nodal analysis (MNA) for the description of electric circuits including lumped elements and refined models describing three-dimensionally resolved electromagnetic fields based on (quasistatic) Maxwell's equations. Such coupled systems arise in various practical situations, e.g., when simulating electric machines or accelerator circuits; see~\cite{Arkkio_1987aa, Bortot_2018ab}. Mathematically, the field--circuit coupling is established by prolongation and restriction operators which can be related to conductor models, e.g., stranded, solid, or foil models \cite{Bedrosian_1993aa,De-Gersem_2001aa,Schops_2013aa,Paakkunainen_2024aa}. All models, i.e., circuit and field, are known to be DAEs of index up to two depending on the circuit topology; see \cite{Estevez-Schwarz_2000aa,Bartel_2011aa, Cortes-Garcia_2019al}. \smallskip

The main contributions of this paper are summarized in the following.
\begin{itemize}
    \item We extend the energy-based representation from \cite{AltS25} by incorporating a possibly singular coefficient matrix in front of the time derivative of a part of the state vector. Moreover, we demonstrate that the structure of this new representation guarantees a corresponding dissipation inequality and is preserved under power-preserving and dissipative interconnections. In addition, we show for a special sub-class with quadratic Hamiltonian that applying the implicit midpoint rule leads to a system with guaranteed dissipation inequality on the time-discrete level.
    \item We consider (linear) finite element models for stranded, solid, and foil conductors and show how to formulate them in terms of the new energy-based representation. Similarly, we demonstrate that a general linear circuit model based on MNA also fits into the framework. Furthermore, we use the energy-based formulations of the single models and the interconnection property of the structure to argue that a coupled system consisting of several conductors and a circuit is also of the same energy-based structure.
    \item We present numerical examples to validate the theoretical considerations. Field--circuit coupled problems with different conductor models are examined to confirm the predictions and to demonstrate applicability towards engineering applications. The open source model implementations are publicly available.
\end{itemize}
\section{Energy-Based Modeling}
\label{sec:theory}

In \cite{AltS25}, the authors have introduced a structured energy-based formulation which is especially suitable for DAEs. This formulation combines different approaches in a hybrid way, namely classical input--state--output pH systems as well as gradient systems. In this paper, we propose a generalization that allows for additional algebraic components in the variables, as this is important for the MNA model to be investigated in \Cref{sec:conductor:MNA}. The new formulation reads
\begin{equation}
\label{eq:moregeneralstructure}
	\begin{bmatrix}
		\nabla_{\state_1}\hamiltonian(\state_1,\state_2)\\
		\Emat\ddt{\state}_2\\
		0
	\end{bmatrix}
	= (\Jmat-\Rmat)
	\begin{bmatrix}
		\ddt{\state}_1\\
		\effort(\state_1,\state_2)\\
		\state_3
	\end{bmatrix}
	+\Bmat\controlinput
\end{equation}
with state $\state = (\state_1,\state_2,\state_3)\in C^1([t_0,\finaltime],\R^{n_1}\times\R^{n_2}\times\R^{n_3})$, sufficiently smooth input $\controlinput\colon [t_0,\finaltime]\to\R^m$, Hamiltonian $\hamiltonian\in C^1(\R^{n_1}\times\R^{n_2})$, effort $\effort\in C(\R^{n_1}\times\R^{n_2},\R^{n_2})$, and coefficient matrices $\Emat\in\R^{n_2\times n_2}$, $\Jmat,\Rmat\in\R^{n\times n}$, $\Bmat\in\R^{n\times m}$ with total dimension $n\vcentcolon= n_1+n_2+n_3$. 
The so-called structure matrix $\Jmat$ describes the energy flux among energy storage elements and is assumed to be skew-symmetric. On the other hand, the dissipation matrix $\Rmat$ models an energy loss or dissipation and is assumed to be symmetric and positive semi-definite. Hence, we have 
\[
    \Jmat = -\Jmat\transpose, \qquad 
    \Rmat = \Rmat\transpose \ge 0.
\]
Finally, $\Bmat$ is called port matrix and includes the input, which can also be used for interconnections of different subsystems. Moreover, we assume that 
\[
    \Emat\transpose \effort(\state_1,\state_2) 
    = \nabla_{\state_2}\hamiltonian(\state_1,\state_2),
\]
which is inspired by the port-Hamiltonian structure introduced in \cite{MehM19} and implies that $\nabla_{\state_2}\hamiltonian(\state_1,\state_2)$ is in the column space of $\Emat\transpose$. In addition, we have an equation for the output $\controloutput\in C([t_0,\finaltime],\R^m)$, given by 
\begin{equation}
    \label{eq:outputEquation}
	\controloutput
	= \Bmat\transpose
	\begin{bmatrix}
		\ddt{\state}_1\\
		\effort(\state_1,\state_2)\\
		\state_3
	\end{bmatrix}.		
\end{equation}
\begin{remark}
The matrix $\Emat$ is not necessarily invertible. However, if $\Emat$ equals the identity, then $\effort(\state_1,\state_2) = \nabla_{\state_2}\hamiltonian(\state_1,\state_2)$ and we regain the basic model from~\cite{AltS25}, namely
\begin{equation}
    \label{eq:generalstructure}
    \begin{bmatrix}
        \nabla_{\state_1}\hamiltonian(\state_1,\state_2)\\
        \ddt{\state}_2\\
        0
    \end{bmatrix}
    = (\Jmat-\Rmat)
    \begin{bmatrix}
        \ddt{\state}_1\\
        \nabla_{\state_2}\hamiltonian(\state_1,\state_2)\\
        \state_3
    \end{bmatrix}
    +\Bmat\controlinput.
\end{equation}
Without the variables $\state_1$ and $\state_3$, system~\eqref{eq:generalstructure} equals a classical input--state--output pH system. On the other hand, with vanishing $\state_2$ and $\state_3$, we obtain gradient systems, cf.~\cite{EggHS21}. The state~$\state_3$ facilitates the inclusion of constraints and is especially useful for DAE models. 
\end{remark}
\begin{remark}
    For simplicity, we assume that the complete state vector is continuously differentiable.
    However, since $\state_3$ only occurs algebraically in the equations and since $\Emat$ and the first block column of $\Jmat-\Rmat$ may have non-trivial kernels, certain components of the state vector do not necessarily have to be continuously differentiable.
\end{remark}
\subsection{Power Balance and Interconnections}
We first prove that systems of the form~\eqref{eq:moregeneralstructure}--\eqref{eq:outputEquation} follow the typical power balance. 
\begin{theorem}[Energy dissipation] 
\label{thm:energy_dissipation}
The energy satisfies $\ddt \hamiltonian \le \langle \controloutput, \controlinput\rangle$, where $\langle\cdot,\cdot\rangle$ denotes the Euclidean inner product. In particular, system~\eqref{eq:moregeneralstructure}--\eqref{eq:outputEquation} is energy dissipative for $\controlinput = 0$. 
\end{theorem}
\begin{proof}
A direct calculation shows
\begin{align*}
	\ddt \hamiltonian
	&= \big\langle \nabla_{\state_1}\hamiltonian(\state_1,\state_2), \tddt{\state}_1 \big\rangle 
	+ \big\langle \effort(\state_1,\state_2), \Emat\tddt{\state}_2 \big\rangle \\
	&= \Bigg\langle \begin{bmatrix} \ddt{\state}_1 \\ \effort(\state_1,\state_2) \\ \state_3 \end{bmatrix}, \begin{bmatrix} \nabla_{\state_1}\hamiltonian(\state_1,\state_2) \\ \Emat \ddt{\state}_2 \\ 0 \end{bmatrix} \Bigg\rangle \\
	&= - \Bigg\langle \begin{bmatrix} \ddt{\state}_1 \\ \effort(\state_1,\state_2) \\ \state_3 \end{bmatrix}, \Rmat \begin{bmatrix} \ddt{\state}_1 \\ \effort(\state_1,\state_2) \\ \state_3 \end{bmatrix} \Bigg\rangle 
	+ 
	\Bigg\langle \begin{bmatrix} \ddt{\state}_1 \\ \effort(\state_1,\state_2) \\ \state_3 \end{bmatrix}, \Bmat \controlinput \Bigg\rangle
	\le \langle \controloutput, \controlinput \rangle.  	
\end{align*}
The special case $\controlinput = 0$ clearly yields $\ddt \hamiltonian \le 0$.
\end{proof}
A second important property is that systems of the form~\eqref{eq:moregeneralstructure}--\eqref{eq:outputEquation} can be coupled, leading again to a system of the same structure. 
\begin{theorem}[Structure-preserving interconnection]
\label{th:interconnection}	
Consider two systems of the form~\eqref{eq:moregeneralstructure}--\eqref{eq:outputEquation}, namely  
\begin{align*}
	\begin{bmatrix}
		\partial_{1} \hamiltonian^{[i]} \\
		\Emat^{[i]}\ddt{\state}^{[i]}_2\\
		0
	\end{bmatrix}
	&= 
	\big( \Jmat^{[i]} - \Rmat^{[i]} \big) 
	\begin{bmatrix}
		\ddt{\state}^{[i]}_1\\
		\effort^{[i]}(\state_1^{[i]},\state_2^{[i]})\\
		\state_3^{[i]}
	\end{bmatrix}
	+ 
	\Bmat^{[i]} \controlinput^{[i]}, \\
	\controloutput^{[i]} 
	&= 
	\Bmat^{[i]\, \top}
	\begin{bmatrix}
		\ddt{\state}^{[i]}_1\\
		\effort^{[i]}(\state_1^{[i]},\state_2^{[i]})\\
		\state_3^{[i]}
	\end{bmatrix}
\end{align*}
with respective states $\state^{[1]}$, $\state^{[2]}$, energy functions $\hamiltonian^{[1]}$, $\hamiltonian^{[2]}$, efforts $\effort^{[1]}$, $\effort^{[2]}$, and the short notation $\partial_{k} \hamiltonian^{[i]} = \nabla_{\state_k^{[i]}}\hamiltonian^{[i]}(\state_1^{[i]},\state_2^{[i]})$. 
Then an interconnection of the form  
\[
	\begin{bmatrix} 
		\controlinput^{[1]} \\ 
		\controlinput^{[2]} 
	\end{bmatrix}
	= 
	\big( \FmatSkew - \FmatSym \big)
	\begin{bmatrix} 
		\controloutput^{[1]} \\ 
		\controloutput^{[2]} 
	\end{bmatrix}
	+ 
	\begin{bmatrix} 
		\tilde \controlinput^{[1]} \\ 
		\tilde \controlinput^{[2]} 
	\end{bmatrix}
\]
with $\FmatSkew=-\FmatSkew\transpose$ and positive semi-definite $\FmatSym=\FmatSym\transpose$ yields again a system of the form~\eqref{eq:moregeneralstructure}--\eqref{eq:outputEquation}. 	
\end{theorem}
\begin{proof}
We define the new states
\[
	\state_1 
	\coloneqq 
	\begin{bmatrix} 
		\state_1^{[1]} \\ 
		\state_1^{[2]}
	\end{bmatrix}, \qquad 
	\state_2 
	\coloneqq 
	\begin{bmatrix} 
		\state_2^{[1]} \\ 
		\state_2^{[2]} 
	\end{bmatrix}, \qquad
	\state_3
	\coloneqq 
	\begin{bmatrix} 
		\state_3^{[1]} \\ 
		\state_3^{[2]} 
	\end{bmatrix}
\]
as well as the in- and outputs 
\[
	\controlinput 
	\coloneqq 
	\begin{bmatrix} 
		\controlinput^{[1]} \\ 
		\controlinput^{[2]} 
	\end{bmatrix}, \qquad
	\tilde \controlinput 
	\coloneqq 
	\begin{bmatrix} 
		\tilde \controlinput^{[1]} \\ 
		\tilde \controlinput^{[2]} 
	\end{bmatrix}, \qquad
	\controloutput 
	\coloneqq 
	\begin{bmatrix} 
		\controloutput^{[1]} \\ 
		\controloutput^{[2]} 
	\end{bmatrix}.
\]
We consider $\hamiltonian(\state_1,\state_2) \coloneqq \hamiltonian^{[1]}(\state_1^{[1]},\state_2^{[1]}) + \hamiltonian^{[2]}(\state_1^{[2]},\state_2^{[2]})$ as the Hamiltonian of the coupled system. Assuming the block structure 
\[
	\Jmat^{[i]}
	= 
	\begin{bmatrix} 
		\Jmat^{[i]}_{11} & \Jmat^{[i]}_{12} & \Jmat^{[i]}_{13} \\ 
		-\Jmat^{[i]\, \top}_{12} & \Jmat^{[i]}_{22} & \Jmat^{[i]}_{23} \\
		-\Jmat^{[i]\, \top}_{13} & -\Jmat^{[i]\, \top}_{23} & \Jmat^{[i]}_{33} 
	\end{bmatrix}\!, \
	\Rmat^{[i]}
	= 
	\begin{bmatrix} 
		\Rmat^{[i]}_{11} & \Rmat^{[i]}_{12} & \Rmat^{[i]}_{13} \\ 
		\Rmat^{[i]\, \top}_{12} & \Rmat^{[i]}_{22} & \Rmat^{[i]}_{23} \\
		\Rmat^{[i]\, \top}_{13} & \Rmat^{[i]\, \top}_{23} & \Rmat^{[i]}_{33} 
	\end{bmatrix}\!, \ 
	\Bmat^{[i]}
	= 
	\begin{bmatrix} 
		\Bmat^{[i]}_1 \\ 
		\Bmat^{[i]}_2 \\
		\Bmat^{[i]}_3 
	\end{bmatrix}\!,
\]
we define the matrices 
\[
	\Jmat
	\coloneqq 
	\begin{bmatrix} 
		\Jmat^{[1]}_{11} & 0 & \Jmat^{[1]}_{12} & 0 & \Jmat^{[1]}_{13} & 0\\
		0 & \Jmat^{[2]}_{11} & 0 & \Jmat^{[2]}_{12} & 0 & \Jmat^{[2]}_{13} \\  
		-\Jmat^{[1]\, \top}_{12} & 0 & \Jmat^{[1]}_{22} & 0 & \Jmat^{[1]}_{23} & 0\\
		0 & -\Jmat^{[2]\, \top}_{12} & 0 & \Jmat^{[2]}_{22} & 0 & \Jmat^{[2]}_{23} \\
		-\Jmat^{[1]\, \top}_{13} & 0 & -\Jmat^{[1]\, \top}_{23} & 0 & \Jmat^{[1]}_{33} & 0\\
		0 & -\Jmat^{[2]\, \top}_{13} & 0 & -\Jmat^{[2]\, \top}_{23} & 0 & \Jmat^{[2]}_{33} 
	\end{bmatrix}\!, \
	\Bmat
	\coloneqq
	\begin{bmatrix} 
		\Bmat^{[1]}_{1} & 0 \\
		0 & \Bmat^{[2]}_{1} \\  
		\Bmat^{[1]}_{2} & 0 \\
		0 & \Bmat^{[2]}_{2} \\
		\Bmat^{[1]}_{3} & 0 \\
		0 & \Bmat^{[2]}_{3} 
	\end{bmatrix},
\]
and $\Rmat$ correspondingly. Since $\Jmat^{[i]}$ is skew-symmetric, so are $\Jmat^{[i]}_{11}$, $\Jmat^{[i]}_{22}$ and $\Jmat^{[i]}_{33}$ for $i=1,2$. Hence, $\Jmat$ is skew-symmetric. Similarly, it can be shown that $\Rmat$ is symmetric and positive semi-definite. 
The two output equations can be combined to
\begin{align*}
	\controloutput
	= 
	\begin{bmatrix} 
		\controloutput^{[1]} \\ 
		\controloutput^{[2]} 
	\end{bmatrix}
	= 
	\Bmat\transpose
	\begin{bmatrix} 
		\ddt{\state}_1 \\ 
		\effort(\state_1,\state_2) \\
		\state_3 
	\end{bmatrix}, \qquad
	\effort(\state_1,\state_2)
	= 
	\begin{bmatrix}
		\effort^{[1]}(\state_1^{[1]},\state_2^{[1]}) \\ 
		\effort^{[2]}(\state_1^{[2]},\state_2^{[2]}) \\
	\end{bmatrix}. 
\end{align*}
With $\Emat \coloneqq \Diag(\Emat^{[1]}, \Emat^{[2]})$, the efforts satisfy 
\[
	\Emat\transpose \effort(\state_1,\state_2)
	=  
	\begin{bmatrix} 
		\Emat^{[1]\, \top} \effort^{[1]}(\state_1^{[1]},\state_2^{[1]}) \\ 
		\Emat^{[2]\, \top} \effort^{[2]}(\state_1^{[2]},\state_2^{[2]}) 
	\end{bmatrix}
	=
	\begin{bmatrix} 
		\partial_{2} \hamiltonian^{[1]} \\ 
		\partial_{2} \hamiltonian^{[2]}
	\end{bmatrix}	
	=
	\nabla_{\state_2}\hamiltonian. 
\]
Together with the interconnection equation, namely $\controlinput = (\FmatSkew - \FmatSym)\, \controloutput + \tilde \controlinput$, we get 
\begin{align*}
	\begin{bmatrix}
		\nabla_{\state_1}\hamiltonian  \\
		\Emat \ddt{\state}_2\\
		0
	\end{bmatrix}
	&=
	\big( \Jmat - \Rmat \big) 
	\begin{bmatrix} 
	\ddt{\state}_1 \\ 
	\effort(\state_1,\state_2)\\ 
	\state_3 
	\end{bmatrix}
	+ \Bmat\, (\FmatSkew - \FmatSym)\, \controloutput 
	+ \Bmat \tilde \controlinput \\ 
	&= 
	\Big( \big(\Jmat+\Bmat\FmatSkew\Bmat\transpose\big) - \big(\Rmat+\Bmat\FmatSym\Bmat\transpose\big) \Big) 
	\begin{bmatrix} 
		\ddt{\state}_1 \\ 
		\effort(\state_1,\state_2)\\ 
		\state_3 
	\end{bmatrix}
	+ \Bmat \tilde \controlinput,	
\end{align*}
which has the form~\eqref{eq:moregeneralstructure}. To see this, note that $\Bmat\FmatSkew\Bmat^T$ is skew-symmetric and $\Bmat\FmatSym\Bmat^T$ symmetric positive semi-definite. 
\end{proof}
\begin{remark}
    \label[remark]{rem:state-dependent_matrices}
    We emphasize that the proofs of \Cref{thm:energy_dissipation,th:interconnection} do not require the matrices $\Emat$, $\Jmat$, $\Rmat$, and $\Bmat$ to be constant, but extend also to the case of state-dependent coefficient matrices.
    In the state-dependent case, the skew-symmetry of $\Jmat$ and the symmetry and positive semi-definiteness of $\Rmat$ have to be satisfied pointwise, see also \cite[Rem.~2.1]{AltS25}.
\end{remark}
\subsection{Time Discretization by the Midpoint Rule}
We consider an equidistant partition of the time interval with step size~$\tau$ and $t_0 = 0$. As usual, $x^k$ denotes the approximation of some variable $x$ at time $t_k\coloneqq k\tau$. Moreover, we use the notation $x^{k+1/2} = \frac12 (x^k + x^{k+1})$ and $\controlinput^{k+1/2} = \controlinput(t^{k+1/2})$ for the input. 
\begin{theorem}[Discrete energy dissipation] 
\label{thm:discrete_energy_dissipation}
Let the Hamiltonian be of the quadratic form
\[
	\hamiltonian(\state_1, \state_2)
	= \frac12\, \langle \state_1, \Mmat_1 \state_1 \rangle + \frac12\, \langle \state_2, \Mmat_2 \state_2 \rangle
\]
with symmetric matrices $\Mmat_1\in\R^{n_1\times n_1}$ and $\Mmat_2\in\R^{n_2\times n_2}$. With the discrete energy $\hamiltonian^k \coloneqq \hamiltonian(\state_1^k,\state_2^k)$, the midpoint scheme applied to~\eqref{eq:moregeneralstructure} satisfies 
\[
	\hamiltonian^{k+1} - \hamiltonian^k 
	\le \tau\, \langle \controloutput^{k+1/2}, \controlinput^{k+1/2} \rangle.
\]
In particular, we have $\hamiltonian^{k+1} \le \hamiltonian^k$ for $\controlinput = 0$.
\end{theorem}
\begin{proof}
The midpoint rule applied to~\eqref{eq:moregeneralstructure} results in the iteration
\begin{subequations}
	\label{eq:midpoint}
	\begin{align}
		\begin{bmatrix}  
			\tau\, \nabla_{\state_1}\hamiltonian^{k+1/2}\\
			\Emat\, (\state_2^{k+1} - \state^k_2) \\ 
			0 \end{bmatrix}
		&= 
		\big( \Jmat - \Rmat \big) 
		\begin{bmatrix} 
			\state^{k+1}_1 - \state^k_1 \\ 
			\tau\, \effort(\state_1^{k+1/2},\state_2^{k+1/2}) \\ 
			\tau\, \state_3^{k+1/2} 
		\end{bmatrix}
		+ 
		\tau \Bmat \controlinput^{k+1/2}, \\
		\tau \controloutput^{k+1/2}
		&= 
		\Bmat\transpose
		\begin{bmatrix} 
			\state^{k+1}_1 - \state^k_1 \\ 
			\tau\, \effort(\state_1^{k+1/2},\state_2^{k+1/2}) \\ 
			\tau\, \state_3^{k+1/2} 
		\end{bmatrix}.
	\end{align}
\end{subequations}
Since the Hamiltonian is quadratic, we have $\nabla_{\state_\ell}\hamiltonian^{k} = \Mmat_\ell \state_\ell^k$, $\ell=1,2$. This implies 
\[
	\Emat\transpose \effort(\state^{k+1/2}_1,\state^{k+1/2}_2) 
	= \nabla_{\state_2}\hamiltonian(\state^{k+1/2}_1,\state^{k+1/2}_2)
	= \nabla_{\state_2}\hamiltonian^{k+1/2}
\]
as well as
\[
	2\hamiltonian^k 
	= \langle \state_1^k, \Mmat_1 \state_1^k \rangle + \langle \state_2^k, \Mmat_2 \state_2^k \rangle
	= \langle \state^k, \nabla_{\state}\hamiltonian^{k}\rangle.
\]
Due to the symmetry of $\Mmat_1$ and $\Mmat_2$, we conclude that
\[
	2\hamiltonian^{k+1} - 2\hamiltonian^k
	= \langle \state^{k+1}, \nabla_{\state}\hamiltonian^{k+1}\rangle 
	- \langle \state^k, \nabla_{\state}\hamiltonian^{k}\rangle 
	= 2\, \big\langle \state^{k+1} - \state^k, \partial_{\state}\hamiltonian^{k+1/2} \big\rangle.
\]
Hence, 
\begin{align*}
	\tau\, \big( \hamiltonian^{k+1} - \hamiltonian^k \big) 
	&= \tau\, \big\langle \state^{k+1} - \state^k, \partial_{\state}\hamiltonian^{k+1/2} \big\rangle \\
	&= \Bigg\langle \begin{bmatrix} \state^{k+1}_1 - \state^k_1 \\ \tau\, \nabla_{\state_2}\hamiltonian^{k+1/2} \\ \tau \state_3^{k+1/2} \end{bmatrix}, \begin{bmatrix} \tau\,\nabla_{\state_1}\hamiltonian^{k+1/2} \\ \state^{k+1}_2 - \state^k_2 \\ 0 \end{bmatrix} \Bigg\rangle \\
	&= \Bigg\langle \begin{bmatrix} \state^{k+1}_1 - \state^k_1 \\ \tau\, \effort(\state^{k+1/2}_1,\state^{k+1/2}_2) \\ \tau \state_3^{k+1/2} \end{bmatrix}, \begin{bmatrix} \tau\nabla_{\state_1}\hamiltonian^{k+1/2} \\ \Emat\, (\state^{k+1}_2 - \state^k_2) \\ 0 \end{bmatrix} \Bigg\rangle.
\end{align*}
Inserting~\eqref{eq:midpoint} finally gives 
\begin{align*}
	\tau\, \big( \hamiltonian^{k+1} - \hamiltonian^k \big)
	&\le \Bigg\langle \begin{bmatrix} \state^{k+1}_1 - \state^k_1 \\ \tau\, \effort(\state^{k+1/2}_1,\state^{k+1/2}_2) \\ \tau \state_3^{k+1/2} \end{bmatrix}, 
	\tau \Bmat \controlinput^{k+1/2} \Bigg\rangle \\
	&= \tau^2\, \big\langle \controloutput^{k+1/2}, \controlinput^{k+1/2} \big\rangle.
	\qedhere
\end{align*}
\end{proof}
\begin{remark}
    Similarly as in \Cref{rem:state-dependent_matrices}, \Cref{thm:discrete_energy_dissipation} may be extended to the case where the matrices $\Emat$, $\Jmat$, $\Rmat$, and $\Bmat$ are state-dependent by using midpoint rule approximations for each of these matrix functions. In this case, all steps of the proof remain the same.
    
\end{remark}
\begin{remark}\label{eq:trapez}
For linear time-invariant systems, the midpoint rule and the trapezoidal rule coincide. Hence, also the trapezoidal rule conserves the energy dissipation for this kind of systems, cf.~the numerical experiments in Section~\ref{sec:numerics}. 
\end{remark}
\section{Conductor Models}
\label{sec:conductor}

In low-frequency numerical computations of electromagnetic fields, so-called conductor models are used to couple conducting regions, e.g., coils, to outside circuitry \cite{Bedrosian_1993aa}.  Three conductor models are commonly distinguished: stranded, solid, and foil conductors, \cite{De-Gersem_2001aa,Schops_2013aa,Paakkunainen_2024aa,Cortes-Garcia_2019al}. We consider the conductor models in the magnetoquasistatic regime of Maxwell's equations, i.e., disregarding displacement currents \cite{Jackson_1998aa}. For that, we choose the (modified) vector potential formulation using $\vec{A}\colon\Omega\times[t_0,\finaltime]\to\R^{3}$ as the main unknown, see \cite{Emson_1988aa}. Imposing homogeneous Dirichlet boundary conditions for simplicity of presentation, this leads to the so-called eddy current problem
\begin{align*}
	\nabla\times\bigl(\nu\, \nabla\times\ensuremath{\vec{A}}\bigr) 
    &= \ensuremath{\vec{J}} && \text{in}\quad \Omega,\\
    \vec{n}\times\vec{A}
    &= 0 && \text{on}\quad \partial\Omega,
\end{align*}
where $\vec{n}$ denotes the outward pointing normal vector, $\nu$ is the magnetic reluctivity, and $\vec{J}\colon\Omega\times[t_0,\finaltime]\to\R^{3}$ equals the total current density. Again, for simplicity, we restrict ourselves to a linear isotropic reluctivity $\nu\in\R_{>0}$. 
The total current density is split into
\begin{equation*}
	\ensuremath{\vec{J}}=\ensuremath{\vec{J}}_\text{c}+\ensuremath{\vec{J}}_\text{s},
\end{equation*}
where $\ensuremath{\vec{J}}_\text{c}$ is the conduction current density, accounting for the ohmic losses, and $\ensuremath{\vec{J}}_\text{s}$ the source current density. The conduction current density is given by Ohm's law $\ensuremath{\vec{J}}_\text{c} = -\sigma\partial_{t}\ensuremath{\vec{A}}${, where $-\partial_{t}\ensuremath{\vec{A}}$ is the electric field strength, and}  in which the electric conductivity may jump
\begin{align*}
	\sigma = \begin{cases}
    \sigma_\text{c} & \text{in } \Omega_\text{c}\subset\Omega\\
   0 & \text{otherwise}
    \end{cases}
\end{align*}
with $\sigma_\text{c}\in\R_{>0}$. Then, the parabolic semi-elliptic curl--curl equation can be written as
\begin{align}\label{eq:field_pde}
	\sigma \partial_{t}\ensuremath{\vec{A}} + \nabla\times\bigl(\nu\,\nabla\times\ensuremath{\vec{A}}\bigr) 
    &= \ensuremath{\vec{J}}_\text{s}& & \text{in}\quad \Omega,\\
    \vec{n}\times\vec{A}&=0 & &\text{on}\quad \partial\Omega
\end{align}
with a given (consistent) initial condition $\vec{A}(\vec{x},0)=\vec{A}_0(\vec{x})$ for $\vec{x}\in\Omega$. Note that \eqref{eq:field_pde} does not define $\vec{A}$ uniquely in non-conducting domains due to gradient fields spanning the nullspace of the curl operator \cite{Jackson_1998aa}. There are different choices to deal with the non-uniqueness and we will follow the tree-cotree approach from \cite{Albanese_1988aa}.

To discuss the conductor models giving rise to the source current density $\vec{J}_\text{s}$, we consider Figure~\ref{fig:domains}. It implicitly encodes the following assumptions from~\cite{Cortes-Garcia_2019al}.
\begin{figure}
    \centering
    \begin{tikzpicture}[inner sep = 0cm]
        \node[anchor=south west, inner sep=0] (image) at (0,0) {%
            \includegraphics[width=7cm]{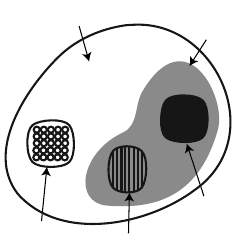}%
        };
        \node[anchor=south] at (2.3cm,6.6cm) {$\Omega$};
        \node[anchor=south] at (6.3cm,6.1cm) {$\Omega_\text{c}$};
        \node[anchor=south] at (3.7cm,-0.2cm) {$\Omega_\text{foil}$};
        \node[anchor=south] at (6.5cm,0.9cm) {$\Omega_\text{sol}$};
        \node[anchor=south] at (1.2cm,0.2cm) {$\Omega_\text{str}$};
    \end{tikzpicture}
    \caption{Computational domain with one representative of each conductor model. }
    \label{fig:domains}
\end{figure}

\begin{assumption}[Domain]\label{as:domain}
$\Omega$ is simply connected and split into two disjoint subdomains $\Omega_{\text{c}}$ and $\Omega_0$ related to the conductivity, i.e., $\overline{\Omega} = \overline{\Omega}_{\text{c}} \cup \overline{\Omega}_{0}$.
Furthermore, there is a source domain $\Omega_{\text{s}}\subset \Omega$, with $\mathrm{supp}(\ensuremath{\vec{J}}_\text{s}) = \Omega_\text{s}$, which can be divided into three disjoint subdomains $\Omega_{\text{s}} = \Omega_{\text{sol}} \cup \Omega_{\text{foil}} \cup \Omega_{\text{str}}$, corresponding to solid, foil, and stranded conductors, respectively. The source subdomains satisfy $\Omega_{\text{sol}},\Omega_{\text{foil}} \subset \Omega_{\text{c}}$ and $\Omega_{\text{str}} \subset \Omega_0$. 
Note that these domains related to conductors may appear $\nstr$, $\nsol$ and $\nfoil$ times. Finally, we assume that all domains are Lipschitz, open, and -- where applicable -- their closures do not intersect, e.g., $\overline{\Omega}_{\text{sol}} \cap \overline{\Omega}_{\text{foil}}=\emptyset$. 
\end{assumption}

A mimetic finite element (FE) discretization of \eqref{eq:field_pde} requires specific function spaces spanned by so-called edge-element basis functions which is well understood, see e.g.~\cite{Monk_2003aa}. Without introducing the corresponding function spaces or basis construction, we directly state the result of semi-discretization by the Galerkin method
\begin{equation}
	\label{eq:field_discrete}
	\begin{aligned}
		\massmatrix \ddt\discretevectorpotential(t)+\stiffnessmatrix\discretevectorpotential(t) &= \discretesourcecurrent(t),
	\end{aligned}
\end{equation}
where $\discretevectorpotential\colon [t_0,\finaltime]\to\R^{\nw}$ contains the FEM coefficients of the magnetic vector potential $\ensuremath{\vec{A}}$ and $\discretesourcecurrent$ is the discretized source current density which will be discussed in detail in the upcoming subsections. The $i$th basis function for $\ensuremath{\vec{A}}$ is denoted as $\wbasisfunction_i$, and the FE matrices $\massmatrix,\stiffnessmatrix \in\R^{\nw\times\nw}$ and $\discretesourcecurrent\in\R^{\nw}$ are defined as
\begin{align}
	[\massmatrix]_{ij} &= \integral{\Omega}{}{\sigma \wbasisfunction_j\cdot \wbasisfunction_i}V,\label{eq:massmatrix}
    \\
    [\stiffnessmatrix]_{ij} &= \integral{\Omega}{}{\nu \nabla\times\wbasisfunction_j\cdot \nabla\times\wbasisfunction_i}V,\label{eq:stiffnessmatrix}
    \\
    \label{eq:j_discrete}
    [\discretesourcecurrent]_i &=\integral{\Omega}{}{\ensuremath{\vec{J}}_\text{s}\cdot \wbasisfunction_i}V. 
\end{align}
Note that \eqref{eq:field_discrete} is a DAE due to the singularity of the conductivity matrix \cite{Nicolet_1996aa}.
\begin{lemma}[System matrices]
\label{lem:sysmatrices}
    The conductivity matrix $\massmatrix$ and the reluctivity matrix $\stiffnessmatrix$ are positive semi-definite.
\end{lemma}
\begin{proof}
    This follows immediately from the Galerkin discretization in \eqref{eq:massmatrix}-\eqref{eq:stiffnessmatrix} and the non-negativity of the material properties.
\end{proof}
The non-uniqueness of the continuous formulation is inherited by the discrete one, i.e., there is a joint nullspace of $\massmatrix$ and $\stiffnessmatrix$. However, this can be removed by gauging, in particular the tree-cotree approach; see \cite[p. 58]{Cor20}. 
\begin{assumption}[Uniqueness]\label{as:Kreg}
We assume that a tree-cotree gauging has been applied such that the matrix pencil $\lambda\massmatrix+\stiffnessmatrix$ with $\lambda\in\R$ is regular, and \eqref{eq:field_discrete} has a unique solution given a consistent initial value $\discretevectorpotential(t_0)=\discretevectorpotential_0\in\R^{\nw}$. 
\end{assumption}
\subsection{Stranded Conductor}
\label{sec:conductor:stranded}

The stranded conductor is a simple homogenization model that avoids resolving single strands by considering a bulk material.  This is a reasonable assumption if their radius is significantly below the skin depth such that eddy current effects are negligible \cite{Bedrosian_1993aa}. For multiple conductors of this type, the source current density has only support in $\Omega_{\mathrm{s}} = \bigcup_{k=1}^{\nstr} \Omega_{\mathrm{str},k}$\, and $ \Omega_{\text{sol}} = \Omega_{\text{foil}} = \emptyset$. The model is based on distribution functions $\vec{\chi}_{\text{str},k}\colon \Omega_{\mathrm{str},k}\to\R^3$, called winding function in \cite{Schops_2013aa} where 
$$
\vec{J}_\text{s}
=\sum_k^{\nstr}\vec{\chi}_{\text{str},k}\,\current_{\text{str},k}
$$
with $\currentvector_\text{str}\colon[t_0,\finaltime]\to\R^{\nstr}$. Each winding function can be discretized as described in \eqref{eq:j_discrete} and collected in  $\Xstr\in\R^{\nw\times\nstr}$. Then, this matrix can also be used to express the voltage drops $\voltagevector_\text{str}\colon[t_0,\finaltime]\to\R^{\nstr}$ such that the resulting system reads \cite{Schops_2013aa}%
\begin{subequations}
	\label{eq:stranded}
	\begin{align}
		\massmatrix \ddt\discretevectorpotential+\stiffnessmatrix\discretevectorpotential-\Xstr\currentvector_\text{str} &=0,\\
		\Xstr\transpose\ddt\discretevectorpotential + \resistancematrix_\text{str}\currentvector_\text{str} - \voltagevector_\text{str} &= 0,
	\end{align}
\end{subequations}
where the resistance is given by
\begin{equation*}
    \resistancematrix_\text{str} = \Xstr\transpose\boldsymbol{M}_\text{str}^{+}\Xstr
\end{equation*}
with the Moore--Penrose pseudo-inverse $\boldsymbol{M}_\text{str}^{+}$ of $\boldsymbol{M}_\text{str}$, which is built analogously to \eqref{eq:massmatrix} by adding a conductivity $\sigma=\sigma_\text{str}>0$ on the stranded conductor domain $\Omega_{\text{str}}$. This accounts for losses due to direct currents, while eddy currents are still neglected.  

\begin{lemma}
\label{lem:resistanceMatrix}
    The resistance matrix $\resistancematrix_\text{str}$ is positive semi-definite.
\end{lemma}
\begin{proof}
    Due to its construction and the fact that $\sigma_\text{str}>0$,  $\boldsymbol{M}_{\text{str}}$ is positive semi-definite and so is $\boldsymbol{M}_{\text{str}}^{+}$.
\end{proof}
In practice, $\resistancematrix_\text{str}$ is positive definite. This can also be proven by exploiting additional properties of the winding functions. However, we only need semi-definiteness for the following result
\begin{theorem}
\label{thm:stranded}
The stranded conductor model~\eqref{eq:stranded} fits into the structure~\eqref{eq:generalstructure}.
\end{theorem}
\begin{proof}
We set $\state_1=\discretevectorpotential$, $\state_3=\currentvector_\text{str}$, $\hamiltonian(\discretevectorpotential) = \tfrac12\,\discretevectorpotential\transpose \stiffnessmatrix\discretevectorpotential$, $\controlinput = \voltagevector_\text{str}$, and 
\begin{equation*}
    \Jmat = 
    \begin{bmatrix}
        0 & \Xstr\\
        -\Xstr\transpose & 0
    \end{bmatrix}
    ,\quad \Rmat = 
    \begin{bmatrix}
        \massmatrix & 0\\
        0 & \resistancematrix_\text{str}
    \end{bmatrix}
    ,\quad \Bmat = 
    \begin{bmatrix}
        0\\
        \boldsymbol{I}
    \end{bmatrix}
    ,
\end{equation*}
where $\boldsymbol{I}$ denotes the identity and $\massmatrix$ and $\resistancematrix_\text{str}$ are positive semi-definite due to Lemma~\ref{lem:sysmatrices} and~\ref{lem:resistanceMatrix}, respectively. 
\end{proof}
Note that all terms in $\Rmat$ are indeed related to conductivity or resistivity. The resulting power balance reads
\begin{equation*}
    \ddt \hamiltonian(\discretevectorpotential)
    = -
    \begin{bmatrix}
        \ddt\discretevectorpotential\\
        \currentvector_\text{str}
    \end{bmatrix}
    \transpose
    \begin{bmatrix}
        \massmatrix & 0\\
        0 & \resistancematrix_\text{str}
    \end{bmatrix}
    \begin{bmatrix}
        \ddt\discretevectorpotential\\
        \currentvector_\text{str}
    \end{bmatrix}
    + \currentvector_\text{str}\cdot\voltagevector_\text{str}.
\end{equation*}
\subsection{Solid Conductor}
\label{sec:conductor:solid}

If eddy currents are non-negligible in a conductor, the solid conductor model is used. Let us assume multiple conductors of this type, then the source current density has support in $\Omega_{\mathrm{s}} = \bigcup_{k=1}^{\nsol} \Omega_{\mathrm{sol},k}$\, and $ \Omega_{\text{str}} = \Omega_{\text{foil}} = \emptyset$. The coupling is established by voltage distribution functions such that 
$$\vec{J}_\text{s}=\sum_{k}^{\nsol}\sigma\vec{\chi}_{\text{sol},k}\voltage_{\text{sol},k}$$
with voltages $\voltagevector_{\text{sol}}\colon[t_0,\finaltime]\to\R^{\nsol}$.
Again, each distribution function $\vec{\chi}_{\text{sol},k}\colon$ $\Omega_{\mathrm{sol},k}\to\R^3$ can be discretized separately and collected in a matrix $\Xsol\in\R^{\nw\times\nsol}$ such that the system reads
\begin{subequations}
	\label{eq:solid}
	\begin{align}
		\massmatrix \ddt\discretevectorpotential+\stiffnessmatrix\discretevectorpotential-\massmatrix\Xsol\voltagevector_{\text{sol}} &=0,\\
		-\Xsol\transpose\massmatrix\transpose\ddt\discretevectorpotential + \conductancematrix_{\text{sol}}\voltagevector_{\text{sol}}-\currentvector_{\text{sol}} &=0
	\end{align}
\end{subequations}
with $\currentvector_{\text{sol}}\colon[t_0,\finaltime]\to\R^{\nsol}$. The (direct current) conductance matrix is defined as
\begin{equation}
\label{eq:conductance_solid}
\conductancematrix_{\text{sol}} =  \Xsol\transpose\massmatrix\Xsol.   
\end{equation}
\begin{remark}
\label{lem:conductanceMatrix}
Due to the positive semi-definiteness of $\massmatrix$ shown in Lemma~\ref{lem:sysmatrices}, the conductance matrix $\conductancematrix_{\text{sol}}$ is positive semi-definite.
\end{remark}
Note that one can show again the stronger result of positive definiteness if considering the exact images and spaces of the involved matrices. However, semi-definiteness is enough to prove the following result. 
\begin{theorem}
\label{thm:solid}
    The solid conductor model~\eqref{eq:solid} fits into the structure~\eqref{eq:generalstructure}.
\end{theorem}
\begin{proof}
We define $\state_1=\discretevectorpotential$, $\state_3=\voltagevector_{\text{sol}}$, $\hamiltonian(\discretevectorpotential) = \tfrac12\,\discretevectorpotential\transpose \stiffnessmatrix\discretevectorpotential$, $\controlinput = \currentvector_{\text{sol}}$, and 
\begin{equation*}
    \Jmat = 0
    ,\quad \Rmat = 
    \begin{bmatrix}
        \massmatrix & -\massmatrix\Xsol\\
        -\Xsol\transpose\massmatrix\transpose & \conductancematrix_{\text{sol}}
    \end{bmatrix}
    ,\quad \Bmat = 
    \begin{bmatrix}
        0\\
        \boldsymbol{I}
    \end{bmatrix}
\end{equation*}
with identity $\boldsymbol{I}$.
Using \eqref{eq:conductance_solid}, we observe that $\Rmat$ may be factorized as
\begin{equation*}
    \Rmat = 
    \begin{bmatrix}
        \boldsymbol{I} & -\Xsol
    \end{bmatrix}
    \transpose\massmatrix
    \begin{bmatrix}
        \boldsymbol{I} & -\Xsol
    \end{bmatrix}
\end{equation*}
and thus $\Rmat$ inherits the positive semi-definiteness from $\massmatrix$, cf.~Lemma~\ref{lem:sysmatrices}.
\end{proof}
Note, all terms in $\Rmat$ are again related to conductivity or resistivity. The power balance reads
\begin{equation*}
    \ddt \hamiltonian(\discretevectorpotential)
    = -
    \begin{bmatrix}
        \ddt\discretevectorpotential\\
        \voltagevector_{\text{sol}}
    \end{bmatrix}
    \transpose
    \begin{bmatrix}
        \massmatrix & -\massmatrix\Xsol\\
        -\Xsol\transpose\massmatrix\transpose & \conductancematrix_{\text{sol}}
    \end{bmatrix}
    \begin{bmatrix}
        \ddt\discretevectorpotential\\
        \voltagevector_{\text{sol}}
    \end{bmatrix}
    +\currentvector_\text{sol}\cdot\voltagevector_\text{sol}.
\end{equation*}
\subsection{Foil Conductor}
\label{sec:conductor:foil}

The foil conductor is represented by a homogenization model that lies, in a sense, between the solid and stranded types. It captures eddy currents within the foils but neglects them in the perpendicular direction \cite{De-Gersem_2001aa, Dular_2002aa}. The following discussion is based on \cite{Paakkunainen_2024aa}, to which we refer for more details. Since the model already resolves several foils, we present here the simplified case of a single conductor to avoid confusion, i.e., $\Omega_{\mathrm{s}}=\Omega_{\mathrm{foil}}$. The source current density is written as
\begin{equation*}
	\ensuremath{\vec{J}}_\text{s} 
    =
    \sigma\Phi\vec{\chi}_\text{sol}, 
\end{equation*}
where the electric field is expressed in terms of a voltage function $\Phi\colon\Omega\times[t_0,\finaltime]\to\R$ and a distribution function $\vec{\chi}_\text{sol}$. By imposing the same current $\current_\text{foil}$ through every (virtual) foil, it can be expressed in the homogenized region as
\begin{equation}
	\label{eq:foil_current} 
    \integral{\Gamma(\alpha)}{}{\sigma\left(-\partial_{t}\vec{A}+ \Phi\vec{\chi}_\text{sol}\right)\cdot\vec{\chi}_\text{sol}}S
    = \frac{\current_\text{foil}}{\foilthickness},
\end{equation}
where $\foilthickness$ is the thickness of a foil and $\Gamma(\alpha)$ defines a surface perpendicular to the thickness in $\Omega_{\mathrm{foil}}$. The voltage over the foil winding $\voltage_\text{foil}$ is the sum of the voltage drops of the individual foils, which is expressed for the homogenized domain as
\begin{equation}
	\label{eq:foil_voltage}
    \voltage_\text{foil} = \frac{1}{\foilthickness} \integral{\coordinatetransform(\alphadomain,\beta,\gamma)}{}{\Phi}s, 
\end{equation}
where $\coordinatetransform(\alphadomain,\beta,\gamma)$ is the one-dimensional domain of homogenization of the foil conductor.
Finally, the FE discretization of \eqref{eq:field_discrete}, \eqref{eq:foil_current}, and \eqref{eq:foil_voltage} yields a system of equations
\begin{subequations}
	\label{eq:foil}
	\begin{align}
		\massmatrix \ddt\discretevectorpotential+\stiffnessmatrix\discretevectorpotential-\Xsigma\discretescalarpotential &=0,\\
		-\Xsigma\transpose\ddt\discretevectorpotential + \conductancematrix_\text{foil}\discretescalarpotential-\cvector\current_\text{foil} &=0,\\
		-\cvector\transpose\discretescalarpotential+\voltage_\text{foil} &=0,
	\end{align}
\end{subequations}
where $\discretescalarpotential\colon[t_0,\finaltime]\to\R^{\np}$ contains the FEM coefficients of the electric scalar potential, and the matrix $\Xsigma\in\R^{\nw\times\np}$ and the vector $\cvector\in\R^{\np}$ are defined by
\begin{align*}
	[\Xsigma]_{i\ell} &= \sum_{j=1}^{\nw}x_j\integral{\Omega}{}{\sigma \pbasisfunction_\ell\wbasisfunction_j\cdot\wbasisfunction_i}V,
    \\
	[\cvector]_k &= \frac{1}{\foilthickness}\integral{\coordinatetransform(\alphadomain,\beta,\gamma)}{}{\pbasisfunction_k}s.
\end{align*}
Here, $\pbasisfunction_k$, $k = 1, \dots, \np$, denote the FEM basis functions for the discretization of $\Phi$, which can be defined independently of the underlying mesh. 
Moreover, the vector $\xvector = \left[x_1,\ldots,x_{\nw}\right]^{\top}$ contains the coordinates of $\vec{\chi}_\text{sol}$ w.r.t.~the basis $(\wbasisfunction_i)_{i=1}^{\nw}$. The conductance matrix $\conductancematrix_\text{foil}\in\R^{\np\times\np}$ is defined as
\begin{equation*}
    \conductancematrix_\text{foil} = \Xsigma^\top \massmatrix^{+} \Xsigma
\end{equation*}
where $\massmatrix^{+}$ is the Moore--Penrose pseudo-inverse of $\massmatrix$.
Our definition of $\conductancematrix_\text{foil}$ ensures consistency with circuit theory, as is discussed in \cite{Paakkunainen_2024aa}. However, the legacy definition from the literature  also fulfills the relevant properties to prove the following Lemma. Note that in the specific case of discretizing $\Phi$ with $\np=1$ basis function $\pbasisfunction_1=1$, the foil conductor model becomes equivalent with the solid conductor model.

\begin{theorem}
\label{thm:foil}
    The foil conductor model~\eqref{eq:foil} fits into the structure~\eqref{eq:generalstructure}.
\end{theorem}
\begin{proof}
We define $\state_1=\discretevectorpotential$, $\state_3=[\discretescalarpotential\transpose\;\,\current_\text{foil}]\transpose$, $\hamiltonian(\discretevectorpotential) = \tfrac12\,\discretevectorpotential\transpose \stiffnessmatrix\discretevectorpotential$, $\controlinput = \voltage_\text{foil}$, and 
\begin{equation*}
    \Jmat = 
    \begin{bmatrix}
        0 & 0 & 0\\
        0 & 0 & \cvector\\
        0 & -\cvector\transpose & 0
    \end{bmatrix}
    ,\quad 
    \Rmat = 
    \begin{bmatrix}
        \massmatrix & -\Xsigma & 0\\
        -\Xsigma\transpose & \conductancematrix_\text{foil} & 0\\
        0 & 0 & 0
    \end{bmatrix}
    ,\quad 
    \Bmat = 
    \begin{bmatrix}
        0\\
        0\\
        1
    \end{bmatrix}
    .
\end{equation*}
It remains to be shown that $\Rmat$ is positive semi-definite. We use \cite[Th.~1.20]{Zha05} which states that this is equivalent to 
\begin{itemize}
\item[(i)] the matrix $\massmatrix$ is positive semi-definite, 
\item[(ii)] the column spaces match, i.e., $C(\Xsigma)\subseteq C(\massmatrix)$,
\item[(iii)] the Schur complement $\conductancematrix_\text{foil}-\Xsigma^\top \massmatrix^{+} \Xsigma$ is positive semi-definite.
\end{itemize}
Points (i) and (iii) follow directly from the definitions together with Lemma~\ref{lem:sysmatrices}. Hence, we only need to show (ii). For this, we denote the projector onto the nullspace of $\massmatrix$ by 
$\boldsymbol{Q}_{\sigma}$. 
It is easy to show that $\boldsymbol{Q}_{\sigma}\Xsigma=0$, see \cite[App.~II]{Paakkunainen_2024aa} which gives (ii) and thus concludes the proof.
\end{proof}
The power balance reads
\begin{equation*}
    \ddt \hamiltonian(\discretevectorpotential)
    = -
    \begin{bmatrix}
        \ddt\discretevectorpotential\\
        \discretescalarpotential
    \end{bmatrix}
    \transpose
    \begin{bmatrix}
        \massmatrix & -\Xsigma\\
        -\Xsigma\transpose & \conductancematrix_\text{foil}
    \end{bmatrix}
    \begin{bmatrix}
        \ddt\discretevectorpotential\\
        \discretescalarpotential
    \end{bmatrix}
    +\current_\text{foil}\voltage_\text{foil}.
\end{equation*}
\subsection{Field--Circuit Coupling using MNA}
\label{sec:conductor:MNA}

There are various formulations to describe electric circuits, however, the arguably most successful is the MNA~\cite{Ho_1975aa}. Using the notation from \cite{Estevez-Schwarz_2000aa}, the dynamics of a linear circuit are then described by the governing equations
\begin{subequations}
    \label{eq:linearMNA}
    \begin{align}
        \Ac\capacitancematrix\Ac\transpose\ddt\vertexpotentials+\Ar\conductancematrix\Ar\transpose\vertexpotentials+\Al\jl+\Av\jv+\Ai\currentvector &= 0,\\
        \inductancematrix\ddt\jl -\Al\transpose\vertexpotentials &=0,\\
        \Av\transpose \vertexpotentials-\voltagevector &=0,
    \end{align}
\end{subequations}
where $\boldsymbol{A}_{\star}\in\R^{(n_{\phi}-1)\times b_{\star}}$ are incidence matrices, $\currentvector\colon[t_0,\finaltime]\to\R^{b_{\text{I}}}$ the source currents and $\voltagevector\colon[t_0,\finaltime]\to\R^{b_{\text{V}}}$ the source voltages. The number of nodes in the circuit is denoted as $n_{\phi}$ and $b_{\star}$ is the number of (branches containing) the examined circuit element. The system unknowns are the node potentials $\vertexpotentials\colon[t_0,\finaltime]\to\R^{n_{\phi}-1}$, currents through inductances $\jl\colon[t_0,\finaltime]\to\R^{b_{\text{L}}}$, and currents through voltage sources $\jv\colon[t_0,\finaltime]\to\R^{b_{\text{V}}}$.

The matrices of conductances $\conductancematrix\in\R^{b_{\text{R}} \times b_{\text{R}}}$, inductances $\inductancematrix\in\R^{b_{\text{L}} \times b_{\text{L}}}$ and capacitances $\capacitancematrix\in\R^{b_{\text{C}} \times b_{\text{C}}}$ are symmetric and positive definite, whereas $\Ac$ does not necessarily have full row rank \cite{Estevez-Schwarz_2000aa}. As a consequence, the matrix $\Ac\capacitancematrix\Ac\transpose$ may, in general, be singular and the problem is of differential--algebraic nature. It is known to have a (differential) index of up to two for specific circuit configurations. Nonetheless, we can show the following result.
\begin{theorem}
\label{thm:MNA}
    The MNA model~\eqref{eq:linearMNA} fits into the structure~\eqref{eq:moregeneralstructure}.
\end{theorem}
\begin{proof}
We set 
\begin{align*}
    \state_2 &= 
    \begin{bmatrix}
        \vertexpotentials\\
        \jl
    \end{bmatrix}
    =\effort(\state_2),\qquad 
    \state_3 
    = \jv,\qquad 
    \controlinput = 
    \begin{bmatrix}
        \currentvector\\
        \voltagevector
    \end{bmatrix}
\end{align*}
and consider the energy $\hamiltonian(\state_2) = \tfrac12\, \vertexpotentials\transpose \Ac\capacitancematrix\Ac\transpose\vertexpotentials + \tfrac12\,\jl\transpose\inductancematrix\jl$. Then the choices
\begin{align*}    
    \Emat &= 
    \begin{bmatrix}
        \Ac\capacitancematrix\Ac\transpose & 0\\
        0 & \inductancematrix
    \end{bmatrix}
    ,\qquad 
    \Jmat = 
    \begin{bmatrix}
        0 & -\Al & -\Av\\
        \Al\transpose & 0 & 0\\
        \Av\transpose & 0 & 0
    \end{bmatrix}
    ,\\
    \Rmat 
    &= 
    \begin{bmatrix}
        \Ar\conductancematrix\Ar\transpose & 0 & 0\\
        0 & 0 & 0\\
        0 & 0 & 0
    \end{bmatrix}
    ,\qquad 
    \Bmat = -
    \begin{bmatrix}
        \Ai & 0\\
        0 & 0\\
        0 & I
    \end{bmatrix}
\end{align*}
satisfy 
\[
     \Emat\transpose\effort(\state_2) 
     = 
     \begin{bmatrix}
        \Ac\capacitancematrix\Ac\transpose & 0\\
        0 & \inductancematrix
    \end{bmatrix} \state_2
    = \begin{bmatrix}
        \Ac\capacitancematrix\Ac\transpose \vertexpotentials \\
        \inductancematrix \jl
    \end{bmatrix}
    = \nabla_{\state_2}\hamiltonian(\state_2)
\]
with $\Rmat$ being again positive semi-definite. 
\end{proof}
We finish this section by exploiting the structure-preserving interconnection property from Theorem~\ref{th:interconnection}.
\begin{theorem}
\label{thm:mainResult}
    The field--circuit coupled model consisting of interconnections 
    of field models \eqref{eq:stranded}, \eqref{eq:solid}, and \eqref{eq:foil} with
    circuit equations \eqref{eq:linearMNA} fits into the structure~\eqref{eq:moregeneralstructure}.
\end{theorem}
\begin{proof}
For the coupling of the MNA equations and the conductor models, we define the currents and voltages in the MNA equations \eqref{eq:linearMNA} as
\begin{equation}
    \label{eq:coupling_MNA}
    \currentvector^\top=
    \begin{bmatrix}
        \currentvector_\text{str}^\top & \current_\text{foil} & \currentvector_\text{src}^\top
    \end{bmatrix}
    ,\qquad \voltagevector^\top=
    \begin{bmatrix}
        \voltagevector_\text{sol}\transpose & \voltagevector_\text{src}^\top
    \end{bmatrix}
    ,
\end{equation}
where $\currentvector_\text{src}$ and $\voltagevector_\text{src}$ may represent external time-dependent current or voltage sources.
For the inputs of the conductor models, we first introduce the splittings
\begin{equation*}
    \Ai = 
    \begin{bmatrix}
        \Astr & \Afoil & \Asrc
    \end{bmatrix}
    ,\qquad \jv\transpose = 
    \begin{bmatrix}
        \jsol\transpose & \jsrc\transpose
    \end{bmatrix}
\end{equation*}
according to the splittings of $\currentvector$ and $\voltagevector$, respectively, and set
\begin{equation}
    \label{eq:coupling_conductors}
    \currentvector_\text{sol} 
    = \jsol,\qquad 
    \voltagevector_\text{str} 
    = \Astr\transpose\vertexpotentials,\qquad 
    \voltage_{\text{foil}} 
    = \Afoil\transpose\vertexpotentials.
\end{equation}
In total, equations~\eqref{eq:coupling_MNA} and \eqref{eq:coupling_conductors} may be summarized as
\begin{equation*}
    \underbrace{
    \begin{bmatrix}
        \currentvector\\
        \voltagevector\\
        \voltagevector_\text{str}\\
        \currentvector_\text{sol}\\
        \voltage_{\text{foil}}
    \end{bmatrix}
    }_{=\controlinput}=\underbrace{
    \begin{bmatrix}
        0 & 0 & 0 & 0 & 0 & \boldsymbol{I} & 0 & 0\\
        0 & 0 & 0 & 0 & 0 & 0 & 0 & 1\\
        0 & 0 & 0 & 0 & 0 & 0 & 0 & 0\\[0.1cm]
        0 & 0 & 0 & 0 & 0 & 0 & \boldsymbol{I} & 0\\
        0 & 0 & 0 & 0 & 0 & 0 & 0 & 0\\[0.1cm]
        -\boldsymbol{I} & 0 & 0 & 0 & 0 & 0 & 0 & 0\\
        0 & 0 & 0 & -\boldsymbol{I} & 0 & 0 & 0 & 0\\
        0 & -1 & 0 & 0 & 0 & 0 & 0 & 0
    \end{bmatrix}
    }_{=\FmatSkew}\underbrace{
    \begin{bmatrix}
        -\Astr\transpose\vertexpotentials\\
        -\Afoil\transpose\vertexpotentials\\
        -\Asrc\transpose\vertexpotentials\\
        -\jsol\\
        -\jsrc\\
        \currentvector_\text{str}\\
        \voltagevector_\text{sol}\\
        \current_{\text{foil}}
    \end{bmatrix}
    }_{=\controloutput}+\underbrace{
    \begin{bmatrix}
        0\\
        0\\
        \currentvector_\text{src}\\
        0\\
        \voltagevector_\text{src}\\
        0\\
        0\\
        0
    \end{bmatrix}
    }_{=\tilde \controlinput}.
\end{equation*}
Since $\FmatSkew$ is skew-symmetric, we can apply Theorem~\ref{th:interconnection}. For this, we use the fact that the single components all fit in the structure~\eqref{eq:moregeneralstructure}--\eqref{eq:outputEquation}, cf.~Theorems~\ref{thm:stranded}, \ref{thm:solid}, \ref{thm:foil} for the conductor models and Theorem~\ref{thm:MNA} for the MNA model.
\end{proof}
\section{Numerical Examples}
\label{sec:numerics}

This section provides numerical examples illustrating the concepts discussed in the previous sections. Section~\ref{sec:numerics_oscillator} discusses an oscillator circuit containing a lumped capacitor and an inductor which embeds a field model. Section~\ref{sec:numerics_index2} contains an example of an index 2 DAE, and in Section~\ref{sec:numerics_transformer}, a three-phase transformer model is examined. The FE models were implemented with the open source software GetDP \cite{Dular_1998ac} and Gmsh \cite{Geuzaine_2009ab}. The model implementations are made available in \cite{AltZ25}.

\subsection{Oscillator Circuit}
\label{sec:numerics_oscillator}

First, we examine the toy example of an oscillator circuit for verification purposes. Figure~\ref{fig:oscillator_example} shows the circuit studied and the corresponding domains of the field problem. The inductor is modeled with both the stranded~\eqref{eq:stranded} and the solid conductor models~\eqref{eq:solid} in a 2D axisymmetric setting (which requires no gauging). This field--circuit coupled model with a stranded or solid conductor fits into the structure \eqref{eq:moregeneralstructure}, as stated in Theorem~\ref{thm:mainResult}. Table~\ref{tab:parameters} lists the parameters of the circuit and the field model. As initial conditions at $t_0=\SI{0}{\second}$, the voltage over the capacitor is set to $\phi(t_0)=v_0=\SI{1}{\volt}$ and there is no current, i.e., $i(t_0)=i_0=\SI{0}{\ampere}$, in the circuit. The simulation is run until $t_f=50\si{\micro\second}$.

\begin{figure}
	\captionsetup{justification=centering}
	\centering
	\subfloat[]{\def\circheight{2.75}
\def\circwidth{1.25}

\begin{circuitikz}
    \useasboundingbox (-2.2cm,-1.5cm) rectangle (2.5cm,1.5cm); 
    \draw[semithick] (0,0) to [short] (-\circwidth,0) to [C=$C$] (-\circwidth,\circheight) to [short](\circwidth,\circheight) to [R, xgeneric] (\circwidth,0) to [short] (0,0);
    
\end{circuitikz}\label{fig:oscillator_circuit}}
    \subfloat[]{\includegraphics{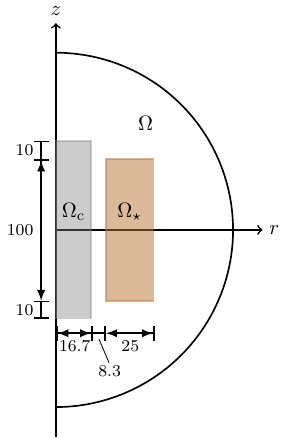}\label{fig:oscillator_fe_geometry}}
    \caption{\protect\subref{fig:oscillator_circuit} Schematic of the oscillator circuit. \protect\subref{fig:oscillator_fe_geometry} 2D axisymmetric domain of the inductor, where $\Omega_{\star}$ denotes either the domain of a solid $\Omega_{\mathrm{sol}}$ or a stranded conductor $\Omega_{\mathrm{str}}$. The dimensions are given in mm.
    }
	\label{fig:oscillator_example}
\end{figure}

\begin{table}
    \caption{Parameters of the oscillator example. \label{tab:parameters}}
    \centering
    \renewcommand{\arraystretch}{1.1}
    \begin{tabular}{c|c}
    \hline 
    Quantity & Value \\
    \hline
    capacitance & \SI{100}{\micro\farad} \\
    core relative permeability & \SI{100}{} \\ 
    conductor relative permeability & \SI{1}{} \\ 
    solid conductor conductivity & \SI{58}{\mega\siemens\per\meter} \\
	time step & \SI{0.1}{\micro\second} \\
    \end{tabular}
\end{table}

When there is no resistance in the stranded conductor model and the core is assumed to be nonconducting, the energy in the oscillator circuit is preserved. An energy conserving time integration scheme is required to ensure this property and we use the trapezoidal rule, see Remark~\ref{eq:trapez}. We mainly examine in detail the trapezoidal rule and the implicit Euler method for time integration, as  they are widely used in practice in engineering applications, but also briefly consider higher-order methods. In the solid conductor model, the nonzero conductivity of the conductor always causes energy dissipation. Figure~\ref{fig:oscillator_noncond_core} shows the oscillation of the energy stored in the capacitor and the inductor when the iron core is modeled as nonconducting. It is seen how the trapezoidal rule is able to conserve the energy in contrast to the implicit Euler method. The energy oscillates with different frequencies for the stranded and solid conductors because the different field models correspond to different lumped inductance values.

\begin{figure}
    \captionsetup{justification=centering}
    \centering
    \subfloat[\footnotesize Stranded conductor Euler]{\begin{tikzpicture}
      \useasboundingbox (-1.1cm,-0.9cm) rectangle (6.1cm,5.7cm);
    \begin{axis}[
        xlabel={Time (\si{\micro\second})},
        ylabel={Energy (\si{\micro\joule})},
        grid,
        width=6.7cm,
        height=5.5cm,
        ymin=-1,
        ymax=53,
        xmin=0,
        xmax=51,
        legend style={at={(1.2,1.38)}, anchor=north,
        font=\footnotesize},
        legend columns=4,
        y label style={xshift=0cm, yshift=-0.4cm, font=\footnotesize},
        x label style={yshift=0.1cm, font=\footnotesize},
        ytick={0,10,20,30,40,50},
        tick label style={font=\footnotesize},
    ]
    
    \addplot[line width=1pt, color=black] table[x expr={1e6*\thisrowno{0}}, y expr={1e6*\thisrowno{1}}] {tikz/data/Etot_osc_str_noncond_core_euler.txt}; \addlegendentry{$\hamiltonian$\phantom{m}}
    \addplot[line width=1pt, color=TUDa-9a] table[x expr={1e6*\thisrowno{0}}, y expr={1e6*\thisrowno{1}}] {tikz/data/EL_osc_str_noncond_core_euler.txt}; \addlegendentry{Magnetic energy\phantom{m}}
    \addplot[line width=1pt, color=TUDa-2a] table[x expr={1e6*\thisrowno{0}}, y expr={1e6*\thisrowno{1}}] {tikz/data/EC_osc_str_noncond_core_euler.txt}; \addlegendentry{Capacitor energy\phantom{m}}
    \addlegendimage{line width=1pt, color=TUDa-7a}
    \addlegendentry{Dissipated energy}

    \end{axis}
\end{tikzpicture}\label{fig:oscillator_str_noncond_core_euler}}
    \subfloat[\footnotesize Stranded conductor trapezoidal]{\begin{tikzpicture}
        \useasboundingbox (-1.1cm,-0.9cm) rectangle (6.1cm,5.7cm);
    \begin{axis}[
        xlabel={Time (\si{\micro\second})},
        ylabel={Energy (\si{\micro\joule})},
        grid,
        width=6.7cm,
        height=5.5cm,
        ymin=-1,
        ymax=53,
        xmin=0,
        xmax=51,
        legend style={at={(0.5,1.8)}, anchor=north, font=\footnotesize},
        y label style={xshift=0cm, yshift=-0.4cm, font=\footnotesize},
        x label style={yshift=0.1cm, font=\footnotesize},
        ytick={0,10,20,30,40,50},
        tick label style={font=\footnotesize},
    ]
    
    \addplot[line width=1pt, color=black] table[x expr={1e6*\thisrowno{0}}, y expr={1e6*\thisrowno{1}}] {tikz/data/Etot_osc_str_noncond_core_trap.txt}; 
    \addplot[line width=1pt, color=TUDa-9a] table[x expr={1e6*\thisrowno{0}}, y expr={1e6*\thisrowno{1}}] {tikz/data/EL_osc_str_noncond_core_trap.txt};
    \addplot[line width=1pt, color=TUDa-2a] table[x expr={1e6*\thisrowno{0}}, y expr={1e6*\thisrowno{1}}] {tikz/data/EC_osc_str_noncond_core_trap.txt};

    \end{axis}
\end{tikzpicture}\label{fig:oscillator_str_noncond_core_trap}}
    \\
    \vspace{-3.5em}
    \subfloat[\footnotesize Solid conductor Euler]{\begin{tikzpicture}
        \useasboundingbox (-1.1cm,-0.9cm) rectangle (6.1cm,5.7cm);
    \begin{axis}[
        xlabel={Time (\si{\micro\second})},
        ylabel={Energy (\si{\micro\joule})},
        grid,
        width=6.7cm,
        height=5.5cm,
        ymin=-1,
        ymax=53,
        xmin=0,
        xmax=51,
        legend style={at={(0.5,1.8)}, anchor=north, font=\footnotesize},
        y label style={xshift=0cm, yshift=-0.4cm, font=\footnotesize},
        x label style={yshift=0.1cm, font=\footnotesize},
        ytick={0,10,20,30,40,50},
        tick label style={font=\footnotesize},
    ]
    
    \addplot[line width=1pt, color=black] table[x expr={1e6*\thisrowno{0}}, y expr={1e6*\thisrowno{1}}] {tikz/data/Etot_osc_sol_noncond_core_euler.txt}; 
    \addplot[line width=1pt, color=TUDa-9a] table[x expr={1e6*\thisrowno{0}}, y expr={1e6*\thisrowno{1}}] {tikz/data/EL_osc_sol_noncond_core_euler.txt};
    \addplot[line width=1pt, color=TUDa-2a] table[x expr={1e6*\thisrowno{0}}, y expr={1e6*\thisrowno{1}}] {tikz/data/EC_osc_sol_noncond_core_euler.txt};
    \addplot[line width=1pt, color=TUDa-7a] table[x expr={1e6*\thisrowno{0}}, y expr={1e6*\thisrowno{1}}] {tikz/data/Ediss_osc_sol_noncond_core_euler.txt};

    \end{axis}
\end{tikzpicture}\label{fig:oscillator_sol_noncond_core_euler}}
    \subfloat[\footnotesize Solid conductor trapezoidal]{\begin{tikzpicture}
        \useasboundingbox (-1.1cm,-0.9cm) rectangle (6.1cm,5.7cm);
    \begin{axis}[
        xlabel={Time (\si{\micro\second})},
        ylabel={Energy (\si{\micro\joule})},
        grid,
        width=6.7cm,
        height=5.5cm,
        ymin=-1,
        ymax=53,
        xmin=0,
        xmax=51,
        legend style={at={(0.5,1.8)}, anchor=north, font=\footnotesize},
        y label style={xshift=0cm, yshift=-0.4cm, font=\footnotesize},
        x label style={yshift=0.1cm, font=\footnotesize},
        ytick={0,10,20,30,40,50},
        tick label style={font=\footnotesize},
    ]
    
    \addplot[line width=1pt, color=black] table[x expr={1e6*\thisrowno{0}}, y expr={1e6*\thisrowno{1}}] {tikz/data/Etot_osc_sol_noncond_core_trap.txt}; 
    \addplot[line width=1pt, color=TUDa-9a] table[x expr={1e6*\thisrowno{0}}, y expr={1e6*\thisrowno{1}}] {tikz/data/EL_osc_sol_noncond_core_trap.txt};
    \addplot[line width=1pt, color=TUDa-2a] table[x expr={1e6*\thisrowno{0}}, y expr={1e6*\thisrowno{1}}] {tikz/data/EC_osc_sol_noncond_core_trap.txt};
    \addplot[line width=1pt, color=TUDa-7a] table[x expr={1e6*\thisrowno{0}}, y expr={1e6*\thisrowno{1}}] {tikz/data/Ediss_osc_sol_noncond_core_trap.txt};

    \end{axis}
\end{tikzpicture}\label{fig:oscillator_sol_noncond_core_trap}}
    \caption{
        The energy of the oscillator with a nonconducting core when
        \protect\subref{fig:oscillator_str_noncond_core_euler} stranded conductor with the implicit Euler method, \protect\subref{fig:oscillator_str_noncond_core_trap} stranded conductor with the trapezoidal rule,  \protect\subref{fig:oscillator_sol_noncond_core_euler} solid conductor with the implicit Euler method, and \protect\subref{fig:oscillator_sol_noncond_core_trap} solid conductor with the trapezoidal rule are used.
        The Hamiltonian contains the contributions from the magnetic and the capacitor energies.}
    \label{fig:oscillator_noncond_core}
\end{figure}

When the core is modeled to be conductive, losses caused by the induced eddy currents in the core are added to the system. A conductivity of \SI{100}{\siemens\per\meter} is assumed in the core region. Figure~\ref{fig:oscillator_cond_core} shows the oscillation of energies in this slightly modified setting. Now, a damping behavior is observed with all the examined simulation cases. However, the additional numerical damping of the implicit Euler method causes a faster damping of the oscillation. When the total dissipated energy is taken into account, it is seen that the energy is still preserved with the trapezoidal rule whereas the implicit Euler method is not able to conserve the energy.

\begin{figure}
    \centering
    \subfloat[\footnotesize Stranded conductor Euler]{\begin{tikzpicture}
    \useasboundingbox (-1.1cm,-0.9cm) rectangle (6.1cm,6cm);
    \begin{axis}[
        xlabel={Time (\si{\micro\second})},
        ylabel={Energy (\si{\micro\joule})},
        grid,
        width=6.7cm,
        height=5.5cm,
        ymin=-1,
        ymax=53,
        xmin=0,
        xmax=51,
        legend style={at={(1.2,1.38)}, anchor=north, font=\footnotesize},
        legend columns=4,
        y label style={xshift=0cm, yshift=-0.4cm, font=\footnotesize},
        x label style={yshift=0.1cm, font=\footnotesize},
        ytick={0,10,20,30,40,50},
        tick label style={font=\footnotesize},
    ]
    
    \addplot[line width=1pt, color=black] table[x expr={1e6*\thisrowno{0}}, y expr={1e6*\thisrowno{1}}] {tikz/data/Etot_osc_str_cond_core_euler.txt}; \addlegendentry{$\hamiltonian$\phantom{m}}
    \addplot[line width=1pt, color=TUDa-9a] table[x expr={1e6*\thisrowno{0}}, y expr={1e6*\thisrowno{1}}] {tikz/data/EL_osc_str_cond_core_euler.txt}; \addlegendentry{Magnetic energy\phantom{m}}
    \addplot[line width=1pt, color=TUDa-2a] table[x expr={1e6*\thisrowno{0}}, y expr={1e6*\thisrowno{1}}] {tikz/data/EC_osc_str_cond_core_euler.txt}; \addlegendentry{Capacitor energy\phantom{m}}
     \addplot[line width=1pt, color=TUDa-7a] table[x expr={1e6*\thisrowno{0}}, y expr={1e6*\thisrowno{1}}] {tikz/data/Ediss_osc_str_cond_core_euler.txt}; \addlegendentry{Dissipated energy}

    \end{axis}
\end{tikzpicture}\label{fig:oscillator_str_cond_core_euler}}
    \subfloat[\footnotesize Stranded conductor trapezoidal]{\begin{tikzpicture}
       \useasboundingbox (-1.1cm,-0.9cm) rectangle (6.1cm,6cm);
    \begin{axis}[
        xlabel={Time (\si{\micro\second})},
        ylabel={Energy (\si{\micro\joule})},
        grid,
        width=6.7cm,
        height=5.5cm,
        ymin=-1,
        ymax=53,
        xmin=0,
        xmax=51,
        legend style={at={(0.5,1.8)}, anchor=north, font=\footnotesize},
        y label style={xshift=0cm, yshift=-0.4cm, font=\footnotesize},
        x label style={yshift=0.1cm, font=\footnotesize},
        ytick={0,10,20,30,40,50},
        tick label style={font=\footnotesize},
    ]
    
    \addplot[line width=1pt, color=black] table[x expr={1e6*\thisrowno{0}}, y expr={1e6*\thisrowno{1}}] {tikz/data/Etot_osc_str_cond_core_trap.txt}; 
    \addplot[line width=1pt, color=TUDa-9a] table[x expr={1e6*\thisrowno{0}}, y expr={1e6*\thisrowno{1}}] {tikz/data/EL_osc_str_cond_core_trap.txt};
    \addplot[line width=1pt, color=TUDa-2a] table[x expr={1e6*\thisrowno{0}}, y expr={1e6*\thisrowno{1}}] {tikz/data/EC_osc_str_cond_core_trap.txt};
    \addplot[line width=1pt, color=TUDa-7a] table[x expr={1e6*\thisrowno{0}}, y expr={1e6*\thisrowno{1}}] {tikz/data/Ediss_osc_str_cond_core_trap.txt};

    \end{axis}
\end{tikzpicture}\label{fig:oscillator_str_cond_core_trap}}
    \\
    \vspace{-3.5em}
    \subfloat[\footnotesize Solid conductor Euler]{\begin{tikzpicture}
       \useasboundingbox (-1.1cm,-0.9cm) rectangle (6.1cm,5.7cm);
    \begin{axis}[
        xlabel={Time (\si{\micro\second})},
        ylabel={Energy (\si{\micro\joule})},
        grid,
        width=6.7cm,
        height=5.5cm,
        ymin=-1,
        ymax=53,
        xmin=0,
        xmax=51,
        legend style={at={(0.5,1.8)}, anchor=north, font=\footnotesize},
        y label style={xshift=0cm, yshift=-0.4cm, font=\footnotesize},
        x label style={yshift=0.1cm, font=\footnotesize},
        ytick={0,10,20,30,40,50},
        tick label style={font=\footnotesize},
    ]
    
    \addplot[line width=1pt, color=black] table[x expr={1e6*\thisrowno{0}}, y expr={1e6*\thisrowno{1}}] {tikz/data/Etot_osc_sol_cond_core_euler.txt}; 
    \addplot[line width=1pt, color=TUDa-9a] table[x expr={1e6*\thisrowno{0}}, y expr={1e6*\thisrowno{1}}] {tikz/data/EL_osc_sol_cond_core_euler.txt};
    \addplot[line width=1pt, color=TUDa-2a] table[x expr={1e6*\thisrowno{0}}, y expr={1e6*\thisrowno{1}}] {tikz/data/EC_osc_sol_cond_core_euler.txt};
    \addplot[line width=1pt, color=TUDa-7a] table[x expr={1e6*\thisrowno{0}}, y expr={1e6*\thisrowno{1}}] {tikz/data/Ediss_osc_sol_cond_core_euler.txt};

    \end{axis}
\end{tikzpicture}\label{fig:oscillator_sol_cond_core_euler}}
    \subfloat[\footnotesize Solid conductor trapezoidal]{\begin{tikzpicture}
       \useasboundingbox (-1.1cm,-0.9cm) rectangle (6.1cm,5.7cm);
    \begin{axis}[
        xlabel={Time (\si{\micro\second})},
        ylabel={Energy (\si{\micro\joule})},
        grid,
        width=6.7cm,
        height=5.5cm,
        ymin=-1,
        ymax=53,
        xmin=0,
        xmax=51,
        legend style={at={(0.5,1.8)}, anchor=north, font=\footnotesize},
        y label style={xshift=0cm, yshift=-0.4cm, font=\footnotesize},
        x label style={yshift=0.1cm, font=\footnotesize},
        ytick={0,10,20,30,40,50},
        tick label style={font=\footnotesize},
    ]
    
    \addplot[line width=1pt, color=black] table[x expr={1e6*\thisrowno{0}}, y expr={1e6*\thisrowno{1}}] {tikz/data/Etot_osc_sol_cond_core_trap.txt}; 
    \addplot[line width=1pt, color=TUDa-9a] table[x expr={1e6*\thisrowno{0}}, y expr={1e6*\thisrowno{1}}] {tikz/data/EL_osc_sol_cond_core_trap.txt};
    \addplot[line width=1pt, color=TUDa-2a] table[x expr={1e6*\thisrowno{0}}, y expr={1e6*\thisrowno{1}}] {tikz/data/EC_osc_sol_cond_core_trap.txt};
    \addplot[line width=1pt, color=TUDa-7a] table[x expr={1e6*\thisrowno{0}}, y expr={1e6*\thisrowno{1}}] {tikz/data/Ediss_osc_sol_cond_core_trap.txt};

    \end{axis}
\end{tikzpicture}\label{fig:oscillator_sol_cond_core_trap}}
    \caption{
        The energy of the oscillator with a conducting core when
        \protect\subref{fig:oscillator_str_cond_core_euler} stranded conductor with the implicit Euler method, \protect\subref{fig:oscillator_str_cond_core_trap} stranded conductor with the trapezoidal rule,  \protect\subref{fig:oscillator_sol_cond_core_euler} solid conductor with the implicit Euler method, and \protect\subref{fig:oscillator_sol_cond_core_trap} solid conductor with the trapezoidal rule are used. The Hamiltonian contains again the contributions from the magnetic and the capacitor energies.}
    \label{fig:oscillator_cond_core}
\end{figure}

Let us investigate the oscillator circuit based on the stranded conductor model \eqref{eq:stranded} without losses, i.e., $\massmatrix=\boldsymbol{0}$ and $\resistancematrix_\text{str}=\boldsymbol{0}$, in more detail. In this case, an equivalent circuit model can be found by eliminating the coefficients of the magnetic vector potential and its closed-form solution $\textbf{z}_\textrm{ref}\transpose=\begin{bmatrix}
\phi_\textrm{ref} & i_\textrm{ref}\end{bmatrix}$ reads
\begin{align*}
\phi_\textrm{ref}(t) &= v_0 \cos(\omega t) + \frac{i_0}{C \, \omega} \sin(\omega t), \\
i_\textrm{ref}(t) &= i_0 \cos(\omega t) - \frac{v_0}{L \, \omega} \sin(\omega t)
\intertext{with (constant) total energy}
\hamiltonian_0 &= \frac{1}{2}Cv_0^2 + \frac{1}{2}Li_0^2,
\end{align*}
where $\omega = 1/\sqrt{LC}$ and $L=\Xstr\transpose\stiffnessmatrix^{+}\Xstr$. 
Equipped with this reference, the numerical solution is examined for different time integration methods. In addition to the previously used implicit Euler method and trapezoidal rule, we apply the second-order backward differentiation formula (BDF2), the Radau~IIA method of order 5, and the Gauss method of order 4 (Gauss4). For details on Radau~IIA and Gauss4, we refer to~\cite{Hairer_1996aa}. Note that only the trapezoidal rule and the Gauss method are energy-conserving.

Figure~\ref{fig:osc_convergence} shows the convergence of the numerical solution towards the reference and the energy conservation for the different time integrators. The errors are measured according to
\begin{align}\label{eq:errors}
\epsilon_{\mathbf{z}} &= \max_k\;\bigl\|
\textbf{z}_\textrm{ref}(t_k)-\textbf{z}_k
\bigr\|_\infty
\quad\text{and}\quad
\epsilon_{\hamiltonian}=\left|\hamiltonian_f - \hamiltonian_0 \right|\,/\,\hamiltonian_0,
\end{align}
where, by slight abuse of notation, $\hamiltonian_k$ and $\textbf{z}_k$ refer to the numerical solution and Hamiltonian evaluated at time $t_k$, respectively.
The expected convergence orders are observed and the energy conserving methods are accurate in respect to the energy up to machine precision. For very small step sizes the convergence of the high order methods saturates due to roundoff errors. Consequently, Radau~IIA never reaches energy conservation up to machine precision.

\begin{figure}
    \centering
    \subfloat[]{\begin{tikzpicture}
      \useasboundingbox (-1.1cm,-0.9cm) rectangle (6.1cm,5.7cm);
    \begin{axis}[
        xlabel={Time step (\si{\second})},
        ylabel={Solution error $\epsilon_{\mathbf{z}}$},
        xmode=log,
        ymode=log,
        grid,
        width=6.7cm,
        height=5.5cm,
        ymax=1000,
        xmin=1.6e-9,
        xmax=1.2e-6,
        legend style={at={(1.2,1.38)}, anchor=north,
        font=\footnotesize},
        legend columns=5,
        y label style={xshift=0cm, yshift=0cm, font=\footnotesize},
        x label style={yshift=0.1cm, font=\footnotesize},
        ytick={1e-15,1e-12,1e-9,1e-6,1e-3,1,1000},
        tick label style={font=\footnotesize},
    ]
    
    \addplot[line width=1pt, color=black,  mark=o] table[x expr={\thisrowno{0}}, y expr={\thisrowno{1}}] {tikz/data/relerr_sol_implicit_euler.txt}; \addlegendentry{Euler\phantom{m}}
    \addplot[line width=1pt, color=TUDa-9a,  mark=square] table[x expr={\thisrowno{0}}, y expr={\thisrowno{1}}] {tikz/data/relerr_sol_bdf2.txt}; \addlegendentry{BDF2\phantom{m}}
    \addplot[line width=1pt, color=TUDa-2a,  mark=triangle] table[x expr={\thisrowno{0}}, y expr={\thisrowno{1}}] {tikz/data/relerr_sol_radau.txt}; \addlegendentry{Radau IIA\phantom{m}}
    \addplot[line width=1pt, color=TUDa-7a,  mark=diamond] table[x expr={\thisrowno{0}}, y expr={\thisrowno{1}}] {tikz/data/relerr_sol_trapezoidal_rule.txt}; \addlegendentry{Trapezoidal\phantom{m}}
    \addplot[line width=1pt, color=TUDa-11a,  mark=x] table[x expr={\thisrowno{0}}, y expr={\thisrowno{1}}] {tikz/data/relerr_sol_gauss_legendre4.txt}; \addlegendentry{Gauss4}

    \begin{scope}
        \begin{scope}
            \draw (axis cs: 5e-8,20) -- (axis cs: 1.25e-8,5) -- node[pos=0.5,anchor=east,font=\tiny,xshift={1mm}]{1} (axis cs: 1.25e-8,20) -- node[pos=0.5,anchor=north,font=\tiny,yshift={3.5mm}]{1} (axis cs: 5e-8,20);
            \draw (axis cs: 5e-8,5e-4) -- (axis cs: 2e-8,5e-4) -- node[pos=0.5,anchor=north,font=\tiny,xshift={0mm}]{1} (axis cs: 5e-8,3.1e-3) -- node[pos=0.5,anchor=west,font=\tiny,xshift={-1mm}]{2} (axis cs: 5e-8,5e-4);
            \draw (axis cs: 9.5e-8,2e-6) -- (axis cs: 5e-8,1.53e-7) -- node[pos=0.5,anchor=east,font=\tiny,xshift={1mm}]{4} (axis cs: 5e-8,2e-6) -- node[pos=0.5,anchor=north,font=\tiny,yshift={3.5mm}]{1} (axis cs: 9.5e-8,2e-6);
            \draw (axis cs: 9.5e-8,1e-11) -- (axis cs: 5e-8,1e-11) -- node[pos=0.5,anchor=north,font=\tiny,yshift={-1.5mm}]{1} (axis cs: 9.5e-8,2.48e-10) -- node[pos=0.5,anchor=west,font=\tiny,xshift={-1mm}]{5} (axis cs: 9.5e-8,1e-11);
        \end{scope}
    \end{scope}

    \end{axis}
\end{tikzpicture}\label{fig:osc_convergence_solution}}
    \subfloat[]{\begin{tikzpicture}
        \useasboundingbox (-1.1cm,-0.9cm) rectangle (6.1cm,5.7cm);
    \begin{axis}[
        xlabel={Time step (\si{\second})},
        ylabel={Conservation error $\epsilon_{\hamiltonian}$},
        xmode=log,
        ymode=log,
        grid,
        width=6.7cm,
        height=5.5cm,
        xmin=1.6e-9,
        xmax=1.2e-6,
        legend style={at={(0.5,1.8)}, anchor=north, font=\footnotesize},
        y label style={xshift=0cm, yshift=0cm, font=\footnotesize},
        x label style={yshift=0.1cm, font=\footnotesize},
        ytick={1e-15,1e-12,1e-9,1e-6,1e-3,1,1000},
        tick label style={font=\footnotesize},
    ]
    
    \addplot[line width=1pt, color=black,  mark=o] table[x expr={\thisrowno{0}}, y expr={\thisrowno{1}}] {tikz/data/relerr_e_implicit_euler.txt};
    \addplot[line width=1pt, color=TUDa-9a,  mark=square] table[x expr={\thisrowno{0}}, y expr={\thisrowno{1}}] {tikz/data/relerr_e_bdf2.txt}; 
    \addplot[line width=1pt, color=TUDa-2a,  mark=triangle] table[x expr={\thisrowno{0}}, y expr={\thisrowno{1}}] {tikz/data/relerr_e_radau.txt};
    \addplot[line width=1pt, color=TUDa-7a,  mark=diamond] table[x expr={\thisrowno{0}}, y expr={\thisrowno{1}}] {tikz/data/relerr_e_trapezoidal_rule.txt};
    \addplot[line width=1pt, color=TUDa-11a,  mark=x] table[x expr={\thisrowno{0}}, y expr={\thisrowno{1}}] {tikz/data/relerr_e_gauss_legendre4.txt}; 
    
    \end{axis}
\end{tikzpicture}\label{fig:osc_convergence_energy}}
    \caption{
        The convergence of 
        \protect\subref{fig:osc_convergence_solution} the solution and \protect\subref{fig:osc_convergence_energy} the energy conservation for different time integrators. Error measures are defined in \eqref{eq:errors}.
        }
    \label{fig:osc_convergence}
\end{figure}

\subsection{Index 2 System}
\label{sec:numerics_index2}

To demonstrate the applicability of our theory for more complicated DAEs, an index 2 system is examined. The oscillator circuit of Section~\ref{sec:numerics_oscillator} is modified by adding a voltage source. Adding a voltage source in parallel to the circuit in Figure~\ref{fig:oscillator_circuit} yields a DAE with index 2 \cite{Cor20}. The simulation parameters are kept the same as in the oscillator example with the exception that now the initial capacitor voltage is set to $v_0=\SI{0}{\volt}$. The voltage source produces a sinusoidal voltage with the amplitude of $\SI{1}{\volt}$ and frequency of $\SI{50}{\kilo\hertz}$. For this subsection, we restrict ourselves to the case of a stranded conductor with a nonconducting iron core.

This model fits also to the structure \eqref{eq:moregeneralstructure}, following again Theorem~\ref{thm:mainResult}. Figure~\ref{fig:index2} shows the energy balance of the system, once with a time-integrator conserving the energy preserving structure (the trapezoidal rule) and once with a non-preserving time-integrator (the implicit Euler method). 
The Euler method yields a solution which violates the energy balance as the total energy brought to the system becomes larger than the Hamiltonian as time progresses.
With the trapezoidal rule the energy brought into the system and the Hamiltonian stay equal. The energy brought into the system $E_{\text{in}}$ is the apparent power of the voltage source $v_{\text{src}}$$i_{\text{src}}$ integrated over time. The numerical integration is not exact but the error is significantly smaller than the energy mismatch of the Euler method as is shown in Figure~{\ref{fig:index2}}.

\begin{figure}
    \centering
    \subfloat[]{\begin{tikzpicture}
      \useasboundingbox (-1.1cm,-0.9cm) rectangle (6.1cm,5.7cm);
    \begin{axis}[
        xlabel={Time (\si{\micro\second})},
        ylabel={Energy (\si{\micro\joule})},
        grid,
        width=6.7cm,
        height=5.5cm,
        ymin=-1,
        ymax=230,
        xmin=0,
        xmax=51,
        legend style={at={(1.2,1.38)}, anchor=north,
        font=\footnotesize},
        legend columns=2,
        y label style={xshift=0cm, yshift=-0.2cm, font=\footnotesize},
        x label style={yshift=0.1cm, font=\footnotesize},
        ytick={0,50,100,150,200},
        tick label style={font=\footnotesize},
    ]
    
    \addplot[line width=1pt, color=black] table[x expr={1e6*\thisrowno{0}}, y expr={1e6*\thisrowno{1}}] {tikz/data/index2_H_euler.txt}; \addlegendentry{$\hamiltonian$\phantom{m}}
    \addplot[line width=1pt, color=TUDa-9a, dashed] table[x expr={1e6*\thisrowno{0}}, y expr={1e6*\thisrowno{1}}] {tikz/data/index2_Ein_euler.txt}; \addlegendentry{$E_{\text{in}}$\phantom{m}}

    \end{axis}
\end{tikzpicture}\label{fig:index2_euler}}
    \subfloat[]{\begin{tikzpicture}
        \useasboundingbox (-1.1cm,-0.9cm) rectangle (6.1cm,5.7cm);
    \begin{axis}[
        xlabel={Time (\si{\micro\second})},
        ylabel={Energy (\si{\micro\joule})},
        grid,
        width=6.7cm,
        height=5.5cm,
        ymin=-1,
        ymax=230,
        xmin=0,
        xmax=51,
        legend style={at={(0.5,1.8)}, anchor=north, font=\footnotesize},
        y label style={xshift=0cm, yshift=-0.2cm, font=\footnotesize},
        x label style={yshift=0.1cm, font=\footnotesize},
        ytick={0,50,100,150,200},
        tick label style={font=\footnotesize},
    ]
    
    \addplot[line width=1pt, color=black] table[x expr={1e6*\thisrowno{0}}, y expr={1e6*\thisrowno{1}}] {tikz/data/index2_H_trap.txt}; 
    \addplot[line width=1pt, color=TUDa-9a, dashed] table[x expr={1e6*\thisrowno{0}}, y expr={1e6*\thisrowno{1}}] {tikz/data/index2_Ein_trap.txt};

    \end{axis}
\end{tikzpicture}\label{fig:index2_trap}}
    \caption{
        The Hamiltonian and the energy supplied by the voltage source for the index 2 model
        \protect\subref{fig:index2_euler} with the implicit Euler method and \protect\subref{fig:index2_trap} with the trapezoidal rule.}
    \label{fig:index2}
\end{figure}

\subsection{Nonlinear Three-Phase Transformer}
\label{sec:numerics_transformer}
As a more realistic example, we examine a 3D model of a three-phase transformer. The model implementation is based on the GetDP model for the Team problem~32 \cite{Guerin17,Bottauscio_2010aa}. The primary and secondary windings are modeled as stranded conductors, and the number of turns is 200 and 100 for the primary and secondary, respectively. The iron core is assumed to be nonconducting, and its nonlinear reluctivity is given as
\begin{equation*}
    \nu\, \bigl(\norm{\vec{B}}\bigr) 
    = 100 + 10\, \operatorname{exp}\bigl({1.8\, \norm{\vec{B}}^{2}}\bigr),
\end{equation*}
where $\norm{\vec{B}}$ is the norm of the magnetic flux density. The classical Newton--Raphson method is carried out each time step until tolerances are met. A tree-cotree gauge is used to ensure the uniqueness of the field solution, see Section~\ref{sec:conductor}. Figure~\ref{fig:transformer_example} shows the three-phase circuit connections and the 3D geometry of the transformer. On the primary side, a line to line root mean square voltage of \SI{1}{\kilo\volt} at \SI{50}{\hertz} frequency is set.

\begin{figure}
	\captionsetup{justification=centering}
	\centering
	\subfloat[]{\scalebox{0.7}{\includegraphics{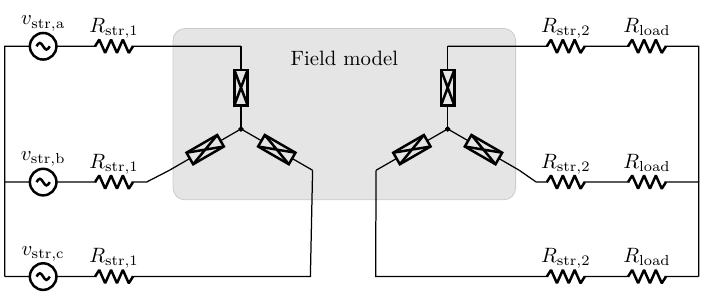}}\label{fig:trafo_circuit}}
    \subfloat[]{\includegraphics[width=0.38\textwidth, trim=-2em 0em 0em 0em, clip]{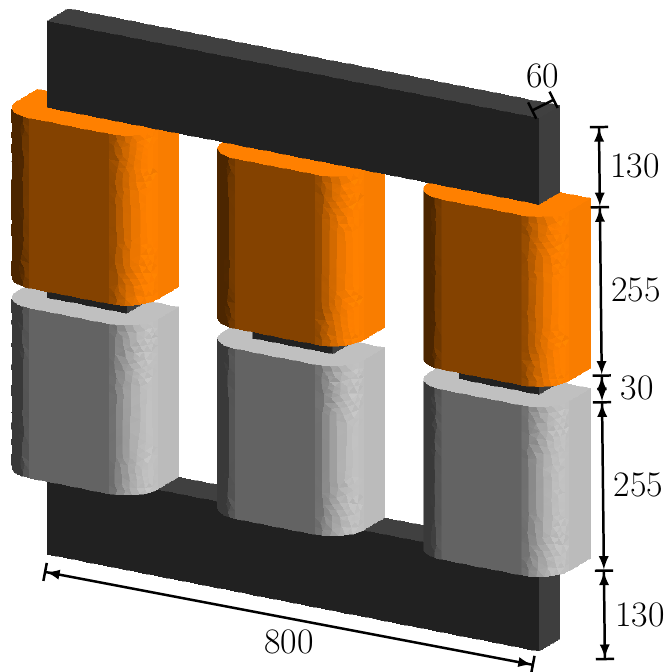}\label{fig:trafo_fe_geometry}}
    \caption{\protect\subref{fig:trafo_circuit} Star connected primary and secondary windings. The resistance values in the circuit are set to $R_{\text{str,1}}=\SI{0.4}{\ohm}$, $R_{\text{str,2}}=\SI{0.2}{\ohm}$, and $R_{\text{load}}=\SI{5}{\ohm}$. \protect\subref{fig:trafo_fe_geometry} 3D modeling domain of the transformer. The dimensions are given in mm.}
	\label{fig:transformer_example}
\end{figure} 

The Hamiltonian of the transformer model is equal to the magnetic energy. Additionally, there is dissipation in the resistors, and the energy brought into the system by voltage sources. The conservation of energy is evaluated by observing the energies brought into and dissipated in the system. The energy brought into the system contains the magnetic energy and the dissipated energy and, consequently, when the magnetic energy is subtracted from the supplied energy, the result is equal to the dissipation in the resistors. This is the criterion to evaluate the energy conservation with the transformer model.

Figure~\ref{fig:trafo_energies} illustrates the energy balance in the transformer when using the midpoint rule and the implicit Euler method. The midpoint rule preserves the total energy well, though not to machine precision. This deviation is not unexpected, as the underlying model violates some of the theoretical assumptions, for instance, due to its nonlinear material characteristics and introduces numerical errors due to quadrature and Newton--Raphson iterations, cf.~\cite{Maier_2025aa}.

The transformer model has a total of \SI{99500}{} degrees of freedom. With a time step of \SI{1}{\milli\second}, the simulation times for the implicit Euler and the midpoint rule were \SI{19}{\minute} and \SI{30}{\minute}, respectively. With the implicit Euler method a total of \SI{229}{} Newton iterations were required whereas the midpoint rule required \SI{321}{}. The simulations were run in a cluster environment with \SI{32}{} threads. We emphasize that the difference in the computation times is not only due to the different time integration. The different model implementations, e.g., the extent to which user-defined and precompiled functions are used, do not allow for a reliable direct comparison of the time integration methods.

\begin{figure}
    \captionsetup{justification=centering}
    \centering
    \begin{tikzpicture}
      \useasboundingbox (-1.4cm,-0.9cm) rectangle (7.0cm,4.3cm);
    \begin{axis}[
        xlabel={Time (\si{\milli\second})},
        ylabel={Energy error (\si{\joule})},
        grid,
        width=8.0cm,
        height=5.5cm,
        ymin=8e-2,
        ymax=500,
        xmin=0,
        xmax=61,
        ymode=log,
        legend style={at={(0.72,0.7)}, anchor=north,
        font=\footnotesize},
        y label style={xshift=0cm, yshift=-0.2cm, font=\footnotesize},
        x label style={yshift=0.1cm, font=\footnotesize},
        tick label style={font=\footnotesize},
    ]
    
    \addplot[line width=1pt, color=black] table[x expr={1e3*\thisrowno{0}}, y expr={abs(\thisrowno{1}-\thisrowno{2})}] {tikz/data/Ediff_trafo_midpoint.txt}; \addlegendentry{Midpoint \phantom{i}}
    \addplot[line width=1pt, color=TUDa-9a] table[x expr={1e3*\thisrowno{0}}, y expr={abs(\thisrowno{1}-\thisrowno{2})}] {tikz/data/Ediff_trafo_euler.txt}; \addlegendentry{Euler \phantom{i}}

    \end{axis}
\end{tikzpicture}
    \caption{Absolute error in the energy of the transformer model when the midpoint rule and the implicit Euler are used for time integration.}
    \label{fig:trafo_energies}
\end{figure}
\section{Summary}
The paper develops a unified modeling and discretization framework for coupled field--circuit systems using an energy-based approach inspired by the pH formalism. It builds upon a recent formulation for DAEs, extending it to handle additional algebraic components in the state variables and enabling the structure-preserving interconnection of electric circuits and electromagnetic field models in a consistent framework. 
The key findings of this paper are: 
\begin{itemize}
    \item
    The considered representation extends a recently introduced energy-based formulation, sharing many of the properties. 
    In particular, it guarantees a dissipation inequality and is preserved under power-preserving and dissipative interconnections. Furthermore, for a particular subclass with quadratic Hamiltonian, the implicit midpoint rule leads to a dissipation inequality on the time-discrete level.
    \item The mentioned subclass encompasses some relevant systems from electrical engineering, including (linear) stranded, solid, and foil conductor models as well as a linear circuit model based on MNA.
    For the latter, the new energy-based representation turns out to be particularly useful.
    The interconnection property of the energy-based structure allows to couple different conductors and circuits while maintaining the structure.
    \item The numerical experiments for an oscillator circuit with stranded and solid conductors illustrate the theoretical findings and show that the implicit midpoint rule leads to a satisfaction of the energy balance on the time-discrete level.
    In contrast, applying the implicit Euler method leads to unphysical results, as expected, with a clear violation of the energy balance.
    \item The numerical experiment of a three-phase transformer demonstrates that the implicit midpoint rule yields also good results in terms of the energy balance when applied to a more realistic system which does not satisfy the assumptions considered in this paper. Again, the implicit Euler method leads to significantly worse results in terms of the energy balance.
    More precisely, the maximum error in the energy balance for the implicit Euler method is more than two orders of magnitude larger than for the implicit midpoint rule.
\end{itemize}
The findings of this paper provide motivation for further research including the derivation of energy-based representations for nonlinear electromagnetic field--circuit coupled models used, e.g., in the design of transformers and electrical machines.

\section*{CRediT authorship contribution statement}
\textbf{Robert Altmann}: Conceptualization, Writing – original draft, Writing – review \& editing. \textbf{Idoia Cortes Garcia}: Conceptualization, Writing – original draft, Writing – review \& editing. \textbf{Elias Paakkunainen}: Conceptualization, Methodology, Software, Visualization, Writing – original draft, Writing – review \& editing. \textbf{Philipp Schulze}: Conceptualization, Writing – original draft, Writing – review \& editing. \textbf{Sebastian Schöps}: Conceptualization, Software, Writing – original draft, Writing – review \& editing.

\section*{Acknowledgments}
This work was partially supported by the joint DFG/FWF Collaborative Research Centre CREATOR (DFG: Project-ID 492661287/TRR 361; FWF: 10.55776/F90) at TU Darmstadt, TU Graz and JKU Linz and by the Graduate School CE within Computational Engineering at the Technical University of Darmstadt. The work of Philipp Schulze is supported by the DFG Collaborative Research Centre 1294: Data Assimilation – The Seamless Integration of Data and Models (Project-ID 318763901).

\bibliographystyle{elsarticle-num} 
\bibliography{bibtex,literature} 
\end{document}